\documentclass{amsart}

\usepackage{amsmath, amssymb, amsthm,enumitem,bm,pict2e,xcolor,hyperref,tikz-cd}
\usepackage{indentfirst}

\hypersetup{
    colorlinks=true,
    linkcolor=blue,
    filecolor=magenta,      
    urlcolor=cyan,
}

\theoremstyle{definition}
\newtheorem{theorem}{Theorem}
\newtheorem*{theorem*}{Theorem}
\numberwithin{theorem}{section}

\newtheorem{proposition}[theorem]{Proposition}
\newtheorem{lemma}[theorem]{Lemma}
\newtheorem{remark}[theorem]{Remark}
\newtheorem{example}[theorem]{Example}
\newtheorem{cor}[theorem]{Corollary}

\newtheorem{innercustomthm}{Theorem}
\newenvironment{customthm}[1]
  {\renewcommand\theinnercustomthm{#1}\innercustomthm}
  {\endinnercustomthm}

\DeclareMathOperator{\initial}{in}

\DeclareMathOperator{\wt}{wt}

\newcommand{\la}{\lambda}

\newcommand{\mL}{\mathcal L}
\newcommand{\mM}{\mathcal M}
\newcommand{\mN}{\mathcal N}

\newcommand{\bC}{\mathbb{C}}
\newcommand{\bP}{\mathbb{P}}
\newcommand{\bR}{\mathbb{R}}
\newcommand{\bZ}{\mathbb{Z}}

\mathcode`l="8000
\begingroup
\makeatletter
\lccode`\~=`\l
\DeclareMathSymbol{\lsb@l}{\mathalpha}{letters}{`l}
\lowercase{\gdef~{\ifnum\the\mathgroup=\m@ne \ell \else \lsb@l \fi}}%
\endgroup

\title[Gr\"obner fans of generalized Hibi ideals and flag varieties]{Gr\"obner fans of Hibi ideals, generalized Hibi ideals and flag varieties}

\author{Igor Makhlin}
\address{I. Makhlin:\newline
Skolkovo Institute of Science and Technology\\ 
Center for Advanced Studies\\
Bolshoy Boulevard 30, bld. 1\\
Moscow 121205\\
Russia\newline
{\it and }\newline
National Research University Higher School of Economics\\
Faculty of Mathematics\\
Ulitsa Usacheva 6\\Moscow 119048\\Russia}
\email{imakhlin@mail.ru}

\begin{document}

\maketitle

\begin{abstract}
The main goal of this paper is to give explicit descriptions of two maximal cones in the Gr\"obner fan of the Pl\"ucker ideal. These cones correspond to the monomial ideals given by semistandard and PBW-semistandard Young tableaux. For the first cone, as an intermediate result we obtain the description of a maximal cone in the Gr\"obner fan of any Hibi ideal. For the second, we generalize the notion of Hibi ideals by associating an ideal with every interpolating polytope. This is a family of polytopes that generalizes the order and chain polytopes of a poset (\`a la Fang--Fourier--Litza--Pegel). We then describe a maximal cone in the Gr\"obner fan of each of these ideals. We also establish some useful facts concerning PBW-semistandardness, in particular, we prove that it provides a new Hodge algebra structure on the Pl\"ucker algebra. 
\end{abstract}

\tableofcontents

\section*{Introduction}

The Gr\"obner fan $\Sigma(I)$ of an ideal $I$ in a polynomial ring is a polyhedral fan composed of cones $C(I,J)$ which parametrize the initial ideals $J$ of $I$. These fans were introduced by Mora and Robbiano in~\cite{MR} as a tool in computer algebra where they can be applied to the construction of universal Gr\"obner bases as well as Gr\"obner basis conversion (\cite{CKM}). However, these fans have also been recognized to play a key role in tropical geometry, since the tropicalization of an algebraic variety can be viewed as the support of a subfan of the Gr\"obner fan (see~\cite[Chapter 2]{MLS}). Gr\"obner fans are also of general relevance to algebraic geometry, since they parametrize an important class of flat degenerations of a scheme, its Gr\"obner degenerations (see, for instance,~\cite[Section 3.3]{HH1}). 

We will be interested in the Gr\"obner fan of the ideal $I$ of Pl\"ucker relations that defines the Pl\"ucker embedding of a (type A) flag variety. The last two decades have seen a variety of results concerning the structure of these fans both in the Grassmannian case (\cite{SS,SW,HJS,KaM}) and the case of a complete or partial flag variety (\cite{BLMM,fafefom,M}). However, all of these papers specifically consider the tropical subfan while little is known regarding the rest of $\Sigma(I)$. This paper was written with the initial goal of making some progress in this direction by explicitly describing two maximal cones in $\Sigma(I)$ that correspond to two known initial monomial ideals of $I$.

Let $I\subset R=\bC[\{X_{i_1<\dots<i_k}\}]$ be the ideal of Pl\"ucker relations for the complete flag variety in the ring of polynomials in Pl\"ucker variables. Consider a poset $\mM$ with elements $a_{i_1,\dots,i_k}$ corresponding to the Pl\"ucker variables with $a_{i_1,\dots,i_k}\le a_{j_1,\dots,i_l}$ if $k\ge l$ and $i_r\le j_r$ for $r\le l$. In this context a standard monomial is a product of Pl\"ucker variables such that the corresponding elements in $\mM$ form a weakly increasing sequence. A classical fact that dates back to the work of Hodge (but see also~\cite{DEP}) is that the monomial ideal $I^m(\mM)$ spanned by non-standard monomials is an initial ideal of $I$. We aim to describe the corresponding maximal cone $C(I,I^m)$ in $\Sigma(I)$ by finding its facets.

It turns out that it is helpful to solve an intermediate problem first. It is easily seen that $\mM$ is a distributive lattice, its Hibi ideal $I^h(\mM)\subset R$ is the ideal generated by the expressions $X_aX_b-X_{a\land b}X_{a\lor b}$. The monomial ideal $I^m(\mM)$ is also an initial ideal of $I^h(\mM)$ and the maximal cone $C(I^h(\mM),I^m(\mM))$ in $\Sigma(I^h(\mM))$ is not too hard to describe. We will term $\{a,b\}\subset\mM$ a \emph{diamond pair} if $a\lor b$ covers both $a$ and $b$. Recall that the Gr\"obner fan is contained in the real space with coordinates corresponding to variables of the ambient polynomial ring. We prove
\begin{customthm}{A}\label{thma}
$C(I^h(\mM),I^m(\mM))$ consists of points $w$ such that $w_a+w_b<w_{a\land b}+w_{a\lor b}$ for every diamond pair $\{a,b\}$. Each of these inequalities provides a facet of the cone.
\end{customthm}
A complete analog of this theorem holds for the Hibi ideal of an arbitrary distributive lattice and we prove it in that generality. Also note that the ideal $I^h(\mM)$ is precisely the toric ideal studied in~\cite{GL,KM} and provides the well known toric degeneration associated with Gelfand--Tsetlin polytope of~\cite{GT}.

General properties of Gr\"obner fans imply that every inequality in the above theorem also provides a facet of $C(I,I^m(\mM))$ and we are to find the remaining facets. They turn out to correspond to certain \emph{special} diamond pairs which can be concisely defined as diamond pairs $\{a,b\}$ for which one has elements $p_1(a,b)$ covered by $a\land b$ and $q_1(a,b)$ covering $a\lor b$ such that $X_aX_b-X_{a\land b}X_{a\lor b}+X_{p_1(a,b)}X_{q_1(a,b)}\in I$ is the corresponding straightening relation.

\begin{customthm}{B}
$C(I,I^m(\mM))$ is cut out by the inequalities in Theorem~\ref{thma} together with $w_a+w_b<w_{p_1(a,b)}+w_{q_1(a,b)}$ for every special diamond pair $\{a,b\}$. Each of these inequalities provides a facet of the cone.
\end{customthm}
This theorem is proved via a careful examination of the combinatorics of $\mM$ and of the straightening relations in $I$.

Now, by definition, $a_{i_1,\dots,i_k}\le a_{j_1,\dots,j_l}$ if and only if the columns $(i_1,\dots,i_k)$ and $(j_1,\dots,j_l)$ form a semistandard Young tableau. A different much more recent notion of semistandardness, called \emph{PBW-semistandardness}, appears in the theory of PBW degenerations (\cite{Fe}). We consider another lattice $\mN$ composed of $b_{\alpha_1,\dots,\alpha_k}$ such that $(\alpha_1,\dots,\alpha_k)$ is a one-column PBW-semistandard tableau and $b_{\alpha_1,\dots,\alpha_k}\ge b_{\beta_1,\dots,\beta_l}$ if $(\alpha_1,\dots,\alpha_k)$ and $(\beta_1,\dots,\beta_l)$ is a two-column PBW-semistandard tableau. The elements of $\mN$ are again in bijection with the Pl\"ucker variables and we consider the monomial ideal $I^m(\mN)\subset R$. Results in the literature can easily be applied to show that $I^m(\mN)$ is an initial ideal of $I$ and we set ourselves the task of describing $C(I,I^m(\mN))$. This requires substantially more work, mainly for two reasons: the combinatorics here is a) far less studied and b) technically more complicated despite many structural similarities.

To find an approach similar to the above one must first find the appropriate analog of $I^h(\mM)$, i.e.\ a toric initial ideal $J$ of $I$ that further degenerates into $I^m(\mN)$. The next task would then be to interpret $J$ as some variation (or, rather, generalization) of the Hibi ideal with respect to the lattice $\mN$. A candidate for the role of $J$ is obvious: the ideal $I^{fflv}$ that provides the toric degeneration associated with the Feigin--Fourier--Littelmann--Vinberg polytope (\cite{FFL1}). $I^{fflv}$ is known to be an initial ideal of $I$ (\cite{favourable,fafefom}) and is also seen to have $I^m(\mN)$ as an initial ideal. The task of interpreting $I^{fflv}$ in terms of the lattice $\mN$ is more difficult.

For a distributive lattice $\mL$ the Hibi ideal $I^h(\mL)$ is known to be the toric ideal associated with the order polytope of the poset $\mathcal P(\mL)$ of join-irreducible elements in $\mL$ (see~\cite{stan,H}). To find a similar interpretation of $I^{fflv}$ we generalize this construction by defining a polytope $\Pi_{U_o,U_c}(P)$ for every partition $U_o\sqcup U_c=P$ of a poset $P$, with $\Pi_{P,\varnothing}(P)$ being the order polytope and $\Pi_{\varnothing,P}(P)$ the chain polytope. Let us note that this construction is very similar to those in~\cite{FF} and~\cite{FFLP}, in fact, our polytopes are unimodular transforms of certain polytopes considered there. Some of the properties we prove (pairwise Ehrhart equivalence and the Minkowski sum property) follow easily from general theorems in~\cite{FFLP} and~\cite{FFP} and should not be viewed as new results.

We then show that for a distributive lattice $\mL$ the toric ideal $I^{U_o,U_c}(\mL)$ associated with $\Pi^{U_o,U_c}(\mathcal P(\mL))$ is generated by the binomials $X_aX_b-X_{a\odot_{U_o,U_c}b}X_{a\lor b}$ for a certain binary operation $\odot_{U_o,U_c}$ that is provided by the polytope. This toric ideal is seen to degenerate into $I^m(\mL)$ and the cone $C(I^{U_o,U_c}(\mL),I^m(\mL))$ is described by a generalization of Theorem~\ref{thma} where $\land$ is replaced by $\odot_{U_o,U_c}$. We then interpret $I^{fflv}$ as an ideal of the form $I^{U_o,U_c}(\mN)$.

In order to augment the description of $C(I^{fflv},I^m(\mN))$ to a description of $C(I,I^m(\mN))$ we establish two fundamental properties of $\mN$. Namely, we construct a lattice isomorphism $\tau:\mM\to\mN$ and show that the Pl\"ucker algebra $R/I$ has the structure of an algebra with straightening laws (or Hodge algebra) on $\mN$ (\cite{H,DEP2}). We note that these facts have been previously conjectured by Xin Fang and, possibly, other experts in the field (see~\cite[Sections 8 and 9]{FL} for some related ideas and conjectures). In this respect we also point out that out of several technical combinatorial proofs that could not be avoided, the proof of the second property is the longest and most complicated in this paper.

We then finally prove our description of $C(I,I^m(\mN))$. Here a special diamond pair $\{a,b\}\subset\mN$ is the $\tau$-image of a special diamond pair in $\mM$ and $X_aX_b-X_{a\odot_{U_o,U_c} b}X_{a\lor b}-X_{g_1(a,b)}X_{h_1(a,b)}\in I$. 
\begin{customthm}{C}
$C(I,I^m(\mN))$ is cut out by the inequalities that provide the facets of $C(I^{fflv},I^m(\mN))$ together with $w_a+w_b<w_{g_1(a,b)}+w_{h_1(a,b)}$ for every special diamond pair $\{a,b\}\subset\mN$. Each of these inequalities provides a facet of the cone.
\end{customthm}

We conclude the paper by giving a certain convex geometric implication of our results that we find curious. It relates our descriptions of the maximal cones to the descriptions of the smaller cones $C(I,I^h(\mM))$ and $C(I,I^{fflv})$ which were given in~\cite{M} and~\cite{fafefom}. 

Let us note that while this paper considers only complete flag varieties, descriptions of the corresponding maximal cones for Grassmannians are seen to follow from our results. The case of a partial flag variety should also be similar but the combinatorial statements would need some adjustments. We also point out that while this paper is primarily concerned with the combinatorics of the Gr\"obner fan, as mentioned above, every face of the cones in question provides a Gr\"obner degeneration of the flag variety. The geometry of these degenerations is the subject of a separate paper currently in preparation.

\textbf{Acknowledgements.} The author would like to thank Xin Fang, Evgeny Feigin and Ievgen Makedonskyi for helpful discussions of these subjects. The work was partially supported by the grant RSF 19-11-00056. This research was also supported in part by the Young Russian Mathematics award.

\section{Gr\"obner fans}

For a positive integer $N$ consider the polynomial ring $R=\bC[X_1,\ldots,X_N]$. A \emph{monomial order} on $R$ is a linear order $\prec$ on the set of monomials in $R$ (products of the $X_i$ with coefficient 1) that respects multiplication: for any monomials $A,B,C\in R$ the relation $A\prec B$ implies $AC\prec BC$. 

For $d=(d_1,\ldots,d_N)\in\bZ_{\ge0}^N$ let $\mathbf X^d=\prod X_i^{d_i}$ and let $(,)$ be the standard scalar product in $\bR^N$. For a set of vectors $d^j\in\bZ_{\ge0}^N$ consider $p=\sum_j c_j\mathbf X^{d^j}\in R$ with $c_j\neq 0$. For a monomial order $\prec$ the \emph{initial part} of $p$ is $\initial_\prec p=c_l\mathbf X^{d^l}$ where $X^{d^l}$ is minimal with respect to $\prec$ among the $X^{d^j}$. For an ideal $I\subset R$ its \emph{initial ideal} $\initial_\prec I$ is the linear span $\bC\{\initial_\prec p, p\in I\}$ which is easily seen to be a monomial ideal in $R$.

We will, however, primarily concern initial ideals of a different kind, given by a vector $w\in\bR^N$ rather than a monomial order. Let $p$ be as above and let $\min_j (w,d^j)=m$, the initial part of $p$ is then \[\initial_w p=\sum_{j|(w,d^j)=m} c_j\mathbf X^{d^j}.\] In other words, we define a grading on $R$ by setting the grading of $X_i$ equal to $w_i$ and then take the nonzero homogeneous component of $p$ of the least possible grading. For an ideal $I\subset R$ its initial ideal $\initial_w I$ is the linear span $\bC\{\initial_w p, p\in I\}$. This is also seen to be an ideal which, however, is not necessarily monomial.

Now let the ideal $I\subset R$ be homogeneous with respect to the usual total degree. For an ideal $J\subset R$ denote $C(I,J)$ the subset in $\bR^N$ of points $w$ for which $\initial_w I=J$. The nonempty subsets $C(I,J)$ form a partition of $\bR^N$ with $w$ contained in $C(I,\initial_w I)$. This partition is known as the \emph{Gr\"obner fan} of $I$, its basic properties in the case of a homogeneous ideal are summed up in the below theorem. 
\begin{theorem}\label{gfans}
For a homogeneous ideal $I\subset R$ the following hold.
\begin{enumerate}[label=(\alph*)]
\item There are only finitely many different nonempty subsets $C(I,J)$.
\item Every nonempty subset $C(I,J)$ is a relatively open polyhedral cone (an open subset of its linear span).
\item Together all the nonempty $C(I,J)$ form a polyhedral fan with support $\bR^N$. This means that every face of the closure $\overline{C(I,J)}$ is itself the closure of some $C(I,J')$. 
\item\label{IJ'J} If $\overline{C(I,J')}$ is a face of $\overline{C(I,J)}$, then the set $C(J',J)$ is nonempty. Conversely, if the sets $C(I,J')$ and $C(J',J)$ are both nonempty, then so is $C(I,J)$ and $\overline{C(I,J')}$ is a face of $\overline{C(I,J)}$.
\item For every monomial order $\prec$ on $R$ there exists $w\in\bR^N$ such that $\initial_\prec I=\initial_w I$. The cone $C(I,\initial_\prec I)$ is maximal in the Gr\"obner fan, i.e.\ it has dimension $N$.
\end{enumerate}
\end{theorem}
\begin{proof}
Modulo a switch between the $\min$ and $\max$ conventions, most of these properties are found in~\cite{S}. Namely, part (a) is Theorem 1.2, (b) is Proposition 2.3, (c) is Proposition 2.4, the first claim in part (e) is Proposition 1.11. The second claim in part (e) is immediate from part (d), since a monomial ideal has no initial ideals other than itself.

Part (d) is also well-known and is easily deduced from Proposition 1.13 in~\cite{S} which states that for any $w,w'\in\bR^N$ for sufficiently small $\varepsilon>0$ one has 
\begin{equation}\label{prop113}
\initial_{w'}(\initial_w I)=\initial_{w+\varepsilon w'} I. 
\end{equation}
If $\overline{C(I,J')}$ is a face of $\overline{C(I,J)}$, then for any $w'\in C(I,J)$, $w\in C(I,J')$ and $\varepsilon>0$ we have $w+\varepsilon w'\in C(I,J)$. By~\eqref{prop113} $J$ is indeed an initial ideal of $J'=\initial_w I$. Conversely, if $w\in C(I,J')$ and $w'\in C(J',J)$, then~\eqref{prop113} provides $J=\initial_{w+\varepsilon w'} I$ for small enough $\varepsilon>0$. Consequently, $C(I,J)$ is nonempty and $\overline{C(I,J)}$ contains $C(I,J')$, hence $\overline{C(I,J')}$ is a face of $\overline{C(I,J)}$.
\end{proof}

A basic fact that is useful in understanding the geometry of Gr\"obner fans is that when the ideal is homogeneous (and we will only consider such ideals), none of the cones in its Gr\"obner fan are pointed (their apexes are not points).
\begin{proposition}
If $I$ is homogeneous, the set $C(I,I)$ is a vector subspace of positive dimension. This subspace is the minimal face (apex) of every nonempty $\overline{C(I,J)}$.
\end{proposition}
\begin{proof}
Part~\ref{IJ'J} of Theorem~\ref{gfans} implies that if $C(I,J)$ is nonempty, then $\overline{C(I,I)}$ is a face of $\overline{C(I,J)}$. This shows that $C(I,I)$ is indeed the unique minimal cone in the fan, in particular, it is necessarily a vector subspace.

Since $I$ is homogeneous, it is evident that adding the same value to every coordinate of $w$ does not alter the ideal $\initial_w I$. This shows that $C(I,I)$ contains the span of $(1,\ldots,1)$. 
\end{proof}

We state one more property of Gr\"obner fans which will be useful to us.
\begin{proposition}\label{minksum}
Suppose that $C(I,J')$ and $C(J',J)$ are both nonempty. Then $C(J',J)$ is the Minkowski sum of $C(I,J)$ and the linear span $\bR C(I,J')$.
\end{proposition}
\begin{proof}
We again make use of Proposition 1.13 in~\cite{S}, see~\eqref{prop113}.
In one direction, consider $w\in C(I,J')$, $w_1\in C(I,J)$ and $w_2\in\bR C(I,J')$. Since $C(I,J')$ is relatively open, for small enough $\varepsilon$ one has $w+\varepsilon w_2\in C(I,J')$ and \[w+\varepsilon(w_1+w_2)\in \varepsilon w_1+C(I,J')\subset C(I,J).\] For $w'=w_1+w_2$ identity~\eqref{prop113} takes the form $\initial_{w_1+w_2} J'=J$.

Conversely, consider $w\in C(I,J')$ and $w'\notin C(I,J)+\bR C(I,J')$. We see that $C(I,J)+\bR C(I,J')$ is invariant under translation by $w$. Therefore, for any $\varepsilon>0$ we have $w+\varepsilon w'\notin C(I,J)+\bR C(I,J')$ and, consequently, $w+\varepsilon w'\notin C(I,J)$. Now~\eqref{prop113} provides $\initial_{w'} J'\neq J$.
\end{proof}

\section{Distributive lattices and Hibi ideals}

For any finite distributive lattice the monomial ideal spanned by non-standard monomials is an initial ideal of the Hibi ideal. In this section we introduce these notions and give an explicit description of the corresponding maximal cone in the Gr\"obner fan of the Hibi ideal.

Let $(\mL,\lor,\land)$ be a finite distributive lattice with induced order relation $<$, so, for instance, $a=b\lor c$ is the minimal (with respect to $<$) element for which $b<a$ and $c<a$. For each $a\in\mL$ introduce the variable $X_a$ and consider the polynomial ring $R(\mL)=\bC[\{X_a,a\in\mL\}]$. The \textit{Hibi ideal} of $\mL$ is the ideal $I^h(\mL)\subset R(\mL)$ generated by the elements \[d(a,b)=X_aX_b-X_{a\lor b}X_{a\land b}\] for all $a,b\in\mathcal L$. Note that $d(a,b)\neq 0$ if and only if $a$ and $b$ are incomparable with respect to $<$. This notion originates from~\cite{H} (it is not to be confused with the related but different ideals introduced in~\cite{HH}, which are referred to by the same name).

Next, let us consider the monomial ideal $I^m(\mL)\subset R(\mL)$ generated by products $X_aX_b$ for all incomparable pairs $\{a,b\}\subset\mathcal L$. Call a monomial $X_{a_1}\ldots X_{a_k}$ \textit{standard} if the elements $a_1,\ldots,a_k$ are pairwise comparable (i.e.\ they form a weak chain). Then one sees that the non-standard monomials are a basis of $I^m(\mL)$. 

When the lattice in consideration is clear from the context we will simply write $R$, $I^h$ and $I^m$ for $R(\mL)$, $I^h(\mL)$ and $I^m(\mL)$. The following fact is easily deduced from~\cite{H}.
\begin{proposition}\label{hibimon}
There exists a monomial order $\prec$ on $R$ such that $\initial_\prec I^h=I^m$.
\end{proposition}
\begin{proof}
Choose a linearization $<_1$ of the order $<$ on $\mL$ (i.e.\ a total order compatible with $<$), let $\le_1$ be the corresponding nonstrict relation and define a graded reverse lexicographic monomial order on $R$ as follows. For tuples $a_1\le_1\ldots\le_1 a_k$ and $b_1\le_1\ldots \le_1 b_l$ set \[X_{a_1}\ldots X_{a_k}\prec X_{b_1}\ldots X_{b_l}\] whenever $k<l$ or $k=l$ and we have $a_i <_1 b_i$ for the largest $i$ such that $a_i\neq b_i$.  

It is evident that for any incomparable $a$ and $b$ we have $\initial_\prec d(a,b)=X_aX_b$. Consequently, $I^m\subset\initial_\prec I^h$. Now, for any homogeneous ideal $I$ let $I_d$ be the subspace of homogeneous polynomials of degree $d$ in $I$. It is well known (see, for instance,~\cite[Corollary 2.4.9]{MLS}) that if $J$ is an initial ideal of $I$, then $\dim I_d=\dim J_d$.  However, it is shown in~\cite[Section 2]{H} that the set of standard monomials in $R$ projects to a basis in $R/I^h$, hence $\dim I^m_d=\dim I^h_d=\dim(\initial_\prec I^h)_d$ for all $d$ and the above inclusion cannot be strict.
\end{proof}

In view of Theorem~\ref{gfans}, this proposition implies that there exist such $w\in\bR^\mL$ that $\initial_w I^h=I^m$ and that the set $C(I^h,I^m)$ of all such $w$ is an open polyhedral cone of dimension $|\mL|$. Below we will give a minimal H-description of this cone, i.e.\ list its facets. 
First, however, let us point out that giving a non-minimal H-description of the cone, i.e.\ expressing it as any intersection of a finite set of half-spaces is rather straightforward. Let $w\in\bR^\mL$ have coordinates $\{w_a,a\in\mL\}$. 
\begin{proposition}\label{redundant}
$C(I^h,I^m)$ consists of the points $w$ satisfying
\begin{equation}\label{diamondineq}
w_a+w_b<w_{a\lor b}+w_{a\land b}
\end{equation}
for all incomparable $a$ and $b$. 
\end{proposition}
\begin{proof}
For $w\in C(I^h,I^m)$ every inequality~\eqref{diamondineq} must hold in order to have $\initial_w d(a,b)=X_aX_b$. Moreover, the intersection of the half-spaces given by~\eqref{diamondineq} for all incomparable pairs is precisely $C(I^h,I^m)$, since for any $w$ in this intersection one has $I^m\subset \initial_w(I^h)$ and, consequently, $\initial_w(I^h)=I^m$. 
\end{proof}
The task at hand is to find the unique maximal irredundant subset of this set of inequalities. 

The following notion will be of key importance to us. We say that $\{a,b\}\subset\mL$ is a \emph{diamond pair} if $a\lor b$ covers both $a$ and $b$ (in terms of the order $<$) and both $a$ and $b$ cover $a\land b$. This means that in the Hasse diagram of the poset $(\mL,<)$ we have the following ``diamond'' as a subgraph:

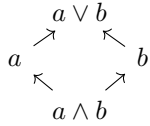
\begin{center}
\begin{tikzcd}[row sep=1mm, column sep=1mm]
&a\lor b&\\
 a  \arrow[ru]&& b\arrow[lu]\\[3pt]
 &\arrow[lu] a\land b\arrow[ru]&
\end{tikzcd}
\end{center}

Note that as a poset $\mL$ is naturally equipped with a grading $|a|$ which is defined by setting $|a_0|=0$ for the unique minimal element $a_0\in\mL$ and $|a|=|b|+1$ when $a$ covers $b$. This is easily verified to be unambiguous. In particular, for a diamond pair $\{a,b\}$ we have \[|a|=|b|=|a\land b|+1=|a\lor b|-1.\] Now to the minimal H-description.

\begin{theorem}\label{hibimonfacets}
The cone $C(I^h,I^m)$ consists of those $w$ that satisfy inequality~\eqref{diamondineq} for all diamond pairs $\{a,b\}$ in $\mL$. This H-description is minimal: for every diamond pair $\{a,b\}$ those points in $\overline{C(I^h,I^m)}$ for which~\eqref{diamondineq} does not hold (equality holds instead) form a facet of $\overline{C(I^h,I^m)}$.
\end{theorem}
\begin{proof}
We are to show two things. First, that every facet of $\overline{C(I^h,I^m)}$ is contained in the hyperplane given by
\begin{equation}\label{diamondeq}
w_a+w_b=w_{a\lor b}+w_{a\land b}
\end{equation}
for some diamond pair $\{a,b\}$. Second, that for any diamond pair $\{a,b\}$ the hyperplane given by~\eqref{diamondeq} contains a facet of $\overline{C(I^h,I^m)}$. 

To prove the first claim choose a facet $F$ of $\overline{C(I^h,I^m)}$, Proposition~\ref{redundant} shows that it is contained in the hyperplane~\eqref{diamondeq} for some incomparable $a$ and $b$. There exists a point $u'\in F$ such that $u'_c+u'_d<u'_{c\lor d}+u'_{c\land d}$ for any incomparable pair $\{c,d\}$ other than $\{a,b\}$ (in fact, $u'$ may be any point in the relative interior of $F$). Consequently, we may choose a point $u\notin\overline{C(I^h,I^m)}$ close enough to $u'$ such that $u_c+u_d<u_{c\lor d}+u_{c\land d}$ for any incomparable pair $\{c,d\}$ other than $\{a,b\}$ but $u_a+u_b>u_{a\lor b}+u_{a\land b}$.

Now, suppose that $\{a,b\}$ is not a diamond pair. This means that the shortest path between $a$ and $b$ in the (non-oriented) Hasse diagram contains more than 2 edges. This lets us consider such a shortest path and choose an element $c\in\mL$ within this path which is not $a$, $b$, $a\lor b$ or $a\land b$.  

We achieve a contradiction by showing that $\dim (\initial_u I^h)_3>\dim I^m_3$. Indeed, let us define a map $\varphi$ from the set of monomials in $I^m_3$ to the set of monomials in $(\initial_u I^h)_3$. Consider the monomial $M=X_p X_q X_r\in (I^m)_3$, if $\{p,q,r\}$ contains as a subset an incomparable pair different from $\{a,b\}$ we set $\varphi(M)=M$. Otherwise we may assume that $p=a$ and $q=b$ (since $\{p,q,r\}$ must contain an incomparable pair) and that $r\ge a\lor b$ or $r\le a\land b$. In this case we set $\varphi(M)=X_{a\lor b}X_{a\land b}X_r$. The map $\varphi$ is easily seen to be injective. However, we also have $\initial_u X_cd(a,b)=X_{a\lor b}X_{a\land b}X_c$ and this monomial does not lie in the image of $\varphi$, since $a\land b<c<a\lor b$. We arrive at the contradiction.

Let us move on to the second claim, it is easily proved with the use of the grading $|a|$. Indeed, the second claim is equivalent to the following. For every diamond pair $\{a,b\}$ there exists a point $v$ such that $v_a+v_b\ge v_{a\lor b}+v_{a\land b}$ but $v_c+v_d<v_{c\lor d}+v_{c\land d}$ for any diamond pair $\{c,d\}\neq\{a,b\}$. Choose a diamond pair $\{a,b\}$, let $f(x)=(x-|a|)^2$ and set $v_a=v_b=1$ and $v_c=f(|c|)$ for all other $c$. The mentioned properties of $v$ are immediate from the convexity of the function $f$.
\end{proof}

\begin{remark}
In terms of the above proof the first claim is equivalent to the following. For every incomparable pair $\{a,b\}$ inequality~\eqref{diamondineq} follows from such inequalities for diamond pairs. This can be then proved directly by expressing~\eqref{diamondineq} as a nonnegative linear combination of the inequalities given by diamond pairs. Such an approach may be more transparent but we were not able to generalize it to a proof of the main Theorem~\ref{maingt} where we use an argument similar to the above instead.
\end{remark}

\begin{remark}
The proof also provides a simple method of constructing points $w\in C(I^h,I^m)$: for any strictly convex function $f$ on $\bR$ set $w_a=f(|a|)$. Of course, not all points in $w\in C(I^h,I^m)$ are obtained in this way.
\end{remark}

\section{The semistandard maximal cone}\label{semistandard}

As we recall below, the monomial ideal spanned by monomials given by non-semistandard Young tableaux is an initial ideal of the ideal of Pl\"ucker relations $I$. The main goal of this section is to describe the corresponding maximal cone in the Gr\"obner fan $\Sigma(I)$. First we recall the definitions.

We fix an integer $n\ge 2$ and consider the set of variables $X_{i_1,\ldots,i_k}$ with $1\le i_1<\ldots<i_k\le n$ and $1\le k\le n-1$, these are known as the \emph{Pl\"ucker variables}. Let $R=\bC[\{X_{i_1,\ldots,i_k}\}]$ be the ring of polynomials in all these variables. We are interested in the \emph{ideal of Pl\"ucker relations} $I\subset R$. Note that $R$ can naturally be viewed as the multigraded coordinate ring of the product \[\bP=\bP(\bC^n)\times\bP(\wedge^2\bC^n)\times\ldots\times\bP(\wedge^{n-1}\bC^n).\] The subvariety cut out by $I$ in this product is the (Pl\"ucker embedding of the) variety of complete flags in $\bC^n$. More information regarding these classical notions can be found in~\cite[Chapter 9]{fulton}.

Now consider a distributive lattice $\mM$ with elements $a_{i_1,\ldots,i_k}$ (one for every Pl\"ucker variable) and the order relation defined as follows. We write $a_{i_1,\ldots,i_k}<a_{j_1,\ldots,j_l}$ whenever $k\ge l$ and $i_r\le j_r$ for every $1\le r\le l$. One easily sees that this indeed defines a distributive lattice and that when $k\ge l$ one has \[a_{i_1,\ldots,i_k}\land a_{j_1,\ldots,j_l}=a_{\min(i_1,j_1),\ldots,\min(i_l,j_l),i_{l+1},\ldots,i_k}\] and \[a_{i_1,\ldots,i_k}\lor a_{j_1,\ldots,j_l}=a_{\max(i_1,j_1),\ldots,\max(i_l,j_l)}.\] It might be helpful to keep in mind the Hasse diagram of this lattice for $n=4$ which can be found in~\cite[Section 14.2]{MS} as well as Example~\ref{hasse} below.

The order relation $<$ is in agreement with the notion of a semistandard Young tableau (see~\cite{fulton}): we have \[a_{i_1^1,\ldots,i_{k_1}^1}\le\ldots\le a_{i_1^m,\ldots,i_{k_m}^m}\] if and only if the Young tableau with columns $(i_1^1,\ldots,i_{k_1}^1),\ldots,(i_1^m,\ldots,i_{k_m}^m)$ is semistandard. Let us write $X_{i_1,\ldots,i_k}$ for $X_{a_{i_1,\ldots,i_k}}$ and identify $R(\mM)$ with $R$. We see that $I^m(\mM)=I^m$ is spanned by monomials \[X_{i_1^1,\ldots,i_{k_1}^1}\ldots X_{i_1^m,\ldots,i_{k_m}^m}\] such that the columns $(i_1^1,\ldots,i_{k_1}^1),\ldots,(i_1^m,\ldots,i_{k_m}^m)$ cannot be arranged into a semistandard Young tableau. 

The following was a pioneering result in the theory of toric degenerations of flag varieties (within this section $I^h=I^h(\mM)$). 
\begin{theorem}[\cite{GL}]
There exists a monomial order $\prec$ on $R$ such that $\initial_\prec I=I^h$.
\end{theorem}
\begin{remark}
It is worth pointing out that the subvariety cut out in $\bP$ by $I^h$ is, in fact, the toric variety associated with the famous Gelfand--Tsetlin polytope (\cite{GT}). This discovery is due to~\cite{KM}.  
\end{remark}

Proposition~\ref{hibimon} and Theorem~\ref{gfans} now imply that the cone $C(I,I^m)\subset\bR^\mM$ is nonempty, we are to find its minimal H-description. It turns out that we already know some (most, in fact, as will be seen below) of the facets from the previous section. 
\begin{proposition}\label{diamondfacets}
For every diamond pair $\{a,b\}\subset\mM$ the hyperplane $w_a+w_b=w_{a\lor b}+w_{a\land b}$ contains a facet of $\overline{C(I,I^m)}$.
\end{proposition}
\begin{proof}
By Proposition~\ref{minksum} the cone $\overline{C(I^m,I^h)}$ is the Minkowski sum of $\bR{C(I,I^h)}$ and $\overline{C(I,I^m)}$. A facet $F$ of this Minkowski sum decomposes into the sum of faces of the summands, i.e a face $H$ of $\overline{C(I,I^m)}$ and $\bR{C(I,I^h)}$ itself. We have $H\subset F$ and we may assume that, in fact, $H=F\cap\overline{C(I,I^m)}$. However the latter intersection contains $\overline{C(I,I^h)}$ and, therefore, $H$ has the same dimension as $H+\bR C(I,I^h)=F$. Therefore $H$ is a facet, i.e.\ every facet of $\overline{C(I^m,I^h)}$ contains a facet of $\overline{C(I,I^m)}$. The proposition now follows from Theorem~\ref{hibimonfacets}.
\end{proof}

We are to determine what facets the cone $\overline{C(I,I^m)}$ has other than those provided by the above proposition. Similarly to the previous section, we first provide a finite set of hyperplanes which contains all hyperplanes providing facets. It is well known that the monomials not contained in $I^m$ (those given by semistandard tableaux) project to a basis in $R/I^m$. This means that for a monomial $M\in I^m$ there exists a unique element in $I$ of the form $M+Q$ where $Q$ is a linear combination of monomials not contained in $I^m$. The elements $M+Q$ are known as the \emph{straightening relations}. In particular, for incomparable $a$ and $b$ in $\mM$ this element can be written as \[s(a,b)=X_aX_b-\sum_{i=0}^{m(a,b)}c_i(a,b)X_{p_i(a,b)}X_{q_i(a,b)}\] where $m(a,b)$ is a positive integer, the pairs $\{p_i(a,b),q_i(a,b)\}$ are pairwise distinct, all $c_i(a,b)\neq 0$ and all $p_i(a,b)<q_i(a,b)$. Since these quadratic expressions form a reduced Gr\"obner basis for both the toric ideal $I^h$ and the monomial ideal $I^m$ (see~\cite[Section 14.3]{MS}), describing them explicitly is a natural and interesting open problem. 

Some basic understanding of what the straightening relations look like comes from the observation that the ideal $I$ (including every straightening relation) and all of its initial ideals are homogeneous with respect to two different gradings. The first grading is a $\bZ^{n-1}$-grading $\deg$ with $\deg(X_{i_1,\ldots,i_k})$ being the $k$th basis vector in $\bZ^{n-1}$. This measures the total degree in variables with exactly $k$ subscripts for all $k$. The homogeneity is immediate from the fact that $I$ is the vanishing ideal of a subvariety in $\bP$.

The second grading $\wt$ is a $\bZ^n$-grading with $\wt X_{i_1,\ldots,i_k}$ being the sum of the $i_r$th coordinate vectors in $\bZ^n$ for $1\le r\le k$. This measures the weight with respect to the action of the maximal torus in $GL_n$.

Consider incomparable $\{a,b\}=\{a_{i_1,\ldots,i_k},a_{j_1,\ldots,j_l}\}$ with $k\ge l$ and any $0\le i\le m(a,b)$. The mentioned homogeneity shows that $p_i(a,b)$ has the form $a_{\alpha_1,\ldots,\alpha_k}$ while $q_i(a,b)$ has the form $a_{\beta_1,\ldots,\beta_l}$. Moreover, the multisets $\{i_1,\ldots,i_k,j_1,\ldots,j_l\}$ and $\{\alpha_1,\ldots,\alpha_k,\beta_1,\ldots,\beta_k\}$ coincide.

The next theorem provides some more detailed information about the straightening relations, first, however, let us recall a helpful fact. Below for any $k\in [1,n-1]$ and $i_1,\dots,i_k\in[1,n]$ we use the standard notation $X_{i_1,\ldots,i_k}=(-1)^\sigma X_{i_{\sigma(1)},\ldots,i_{\sigma(k)}}$ for any permutation $\sigma\in S_k$ (in particular,  $X_{i_1,\ldots,i_k}=0$ when the $i_j$ are not pairwise distinct).
\begin{lemma}\label{k>l}
Consider integers $1\le l<k\le n$ and suppose that \[\sum_r \alpha_r X_{j_1^r,\dots,i_k^r}X_{i_1^r,\dots,j_l^r}\in I.\] Then for any $\alpha_1,\dots,\alpha_{k-l}\in[1,n]$ we also have 
\begin{equation}
\sum_r \alpha_r X_{i_1^r,\dots,i_k^r}X_{j_1^r,\dots,j_l^r,\alpha_1,\dots,\alpha_{k-l}}\in I.
\end{equation}
\end{lemma}
\begin{proof}
For all $1\le i,j\le n$ consider variables $z_{i,j}$. A classical result characterizes $I$ as the kernel of the homomorphism $\pi$ from $R$ to $\bC[\{z_{i,j}\}]$ mapping $X_{i_1,\dots,i_k}$ to the determinant of the $k\times k$ matrix $A$ with elements $A_{i,j}=z_{i,i_j}$. Now expand the determinant $\pi(X_{i_1^r,\dots,j_l^r,\alpha_1,\dots,\alpha_{k-l}})$ with respect to the last $l$ columns (which are composed of the variables $z_{i,\alpha_j}$). Each of the resulting $k\choose {k-l}$ summands is obtained from $\pi(X_{j_1^r,\dots,j_l^r})\pi(X_{\alpha_1,\dots,\alpha_{k-l}})$ by a change of variables $z_{i,j}\mapsto z_{\sigma(i),j}$ for some permutation $\sigma\in S_k$. Note that this change of variables multiplies $\pi(X_{i_1^r,\dots,i_k^r})$ by $(-1)^\sigma$. The result now follows if we group together the summands corresponding to the same change of variables for all $r$. 
\end{proof}

\begin{theorem}\label{strlaws}
\hfill
\begin{enumerate}[label=(\alph*)]
\item The quadratic straightening relations $s(a,b)$ form a minimal generating set of $I$.
\item WLOG we may assume that $p_0(a,b)=a\land b$, $q_0(a,b)=a\lor b$ and $c_0(a,b)=1$.
\item For $1\le i\le m(a,b)$ we have $p_i(a,b)<a\land b$ and $q_i>a\lor b$.
\end{enumerate}
\end{theorem}
\begin{proof}
Parts (a) and (b) are well-known results, see, for instance, \cite[Section 14.3]{MS}. Let $a=a_{i_1,\ldots,i_k}$ and $b=a_{j_1,\ldots,j_l}$ with $k\ge l$. In the case of $k=l$ part (c) stems back to the work of William Hodge, see~\cite{DEP} for a more modern discussion. The case $k>l$ is also known and can be extracted from~\cite{LMS} and, perhaps, even earlier work. However, since the latter paper is written in very different terms, let us show how the case $k>l$ is deduced with the use of Lemma~\ref{k>l}.

Suppose that $k>l$. The relation $s(a,b)$ is independent of $n$ as long as $k,l<n$, all $i_r\le n$ and all $j_r\le n$ (this, for instance, follows from the above characterization in terms of determinants). We may, therefore, assume that all $i_r$ and all $j_r$ are less than $n-k+l$. Now apply Lemma~\ref{k>l} to the relation $s(a,b)$ with the $k-l$ subscripts being added given $\alpha_i=n-k+l+i$. It is easily seen that all the monomials in the resulting relation other than $X_{i_1,\ldots,i_k}X_{j_1,\ldots,j_l,\alpha_1,\dots,\alpha_{k-l}}$ are standard, in other words, the resulting relation is precisely $s(a_{i_1,\ldots,i_k},a_{j_1,\ldots,j_l,\alpha_1,\dots,\alpha_{k-l}})$. In particular, we see that \[p_i(a,b)=p_i(a_{i_1,\ldots,i_k},a_{j_1,\ldots,j_l,\alpha_1,\dots,\alpha_{k-l}})< a_{i_1,\ldots,i_k}\land a_{j_1,\ldots,j_l,\alpha_1,\dots,\alpha_{k-l}}=a\land b.\] We also have \[q_i(a_{i_1,\ldots,i_k},a_{j_1,\ldots,j_l,\alpha_1,\dots,\alpha_{k-l}})>a_{i_1,\ldots,i_k}\lor a_{j_1,\ldots,j_l,\alpha_1,\dots,\alpha_{k-l}}\] which implies $q_i(a,b)>a\lor b$. That is since if $q_i(a,b)=a_{\beta_1,\dots,\beta_l}$, then \[q_i(a_{i_1,\ldots,i_k},a_{j_1,\ldots,j_l,\alpha_1,\dots,\alpha_{k-l}})=a_{\beta_1,\dots,\beta_l,\alpha_1,\dots,\alpha_{k-l}}\] and $a_{i_1,\ldots,i_k}\lor a_{j_1,\ldots,j_l,\alpha_1,\dots,\alpha_{k-l}}$ is similarly determined by $a\lor b$.
\end{proof}
\begin{remark}
Part (c) of the above theorem establishes that the Pl\"ucker algebra $R/I$ is an algebra with straightening laws (or a Hodge algebra). This will be discussed in more detail in Section~\ref{ASL}.
\end{remark}

Next we prove
\begin{proposition}\label{ssytredundant}
The cone $C(I,I^m)$ is composed of all $w$ satisfying 
\begin{equation}\label{strineq}
w_a+w_b<w_{p_i(a,b)}+w_{q_i(a,b)}
\end{equation}
for all incomparable pairs $\{a,b\}$ and all $0\le i\le m(a,b)$.
\end{proposition}
\begin{proof}
For $w\in C(I,I^m)$ we must have $\initial_w s(a,b)=X_aX_b$ which implies inequality~\eqref{strineq}. If all inequalities~\eqref{strineq} hold, then $\initial_w I$ contains $I^m$ and, subsequently, $\initial_w I=I^m$.
\end{proof}

Finding the minimal H-description means finding the unique maximal irredundant subset of the set of inequalities of form~\eqref{strineq}. From Proposition~\ref{diamondfacets} we already know that when $\{a,b\}$ is a diamond pair and $i=0$ the corresponding inequality is contained in this subset. We also know that when $\{a,b\}$ is not a diamond pair and $i=0$, the inequality is not contained therein. To find out what happens when $i\ge 1$ let us describe the diamond pairs in $\mM$ explicitly. The following fact follows directly from the definitions.
\begin{proposition}\label{cover}
If $a_{i_1,\ldots,i_k}$ is covered by $a_{j_1,\ldots,j_l}$, one of the two holds.
\begin{enumerate}
\item $k=l$ and $i_r=j_r$ for all $1\le r\le k$ except exactly one value $r_1$ for which $i_r=j_r-1$.
\item $k=l+1$ and $i_r=j_r$ for all $1\le r\le l$ while $i_k=n$.
\end{enumerate}
\end{proposition}
This directly implies the following.
\begin{proposition}\label{diamondpairs}
If $\{a_{i_1,\ldots,i_k},a_{j_1,\ldots,j_l}\}$ is a diamond pair, one of the two holds.
\begin{enumerate}
\item $k=l$ and $i_r=j_r$ for all $1\le r\le k$ except exactly two values $r_1<r_2$ for which $i_{r_1}=j_{r_1}-1$ and $i_{r_2}=j_{r_2}+1$.
\item $k=l+1$, $i_k=n$ and $i_r=j_r$ for all $1\le r\le l$ except exactly one value $r_1$ for which $i_{r_1}=j_{r_1}+1$.
\end{enumerate}
\end{proposition}

Let us give two small examples which illustrate the two possibilities in the above proposition as well as those in Proposition~\ref{strrels} where we will explicitly describe the straightening relation for a diamond pair.
\begin{example}\label{ex1}
Let $n=4$, $k=l=2$, $(i_1,i_2)=(1,4)$ and $(j_1,j_2)=(2,3)$. This is possibility (1) in Proposition~\ref{diamondpairs} with $r_1=1$ and $r_2=2$. The straightening relation is \[X_{1,4}X_{2,3}-X_{1,3}X_{2,4}+X_{1,2}X_{3,4}\] (which is the defining relation of the Grassmannian $\mathrm{Gr}(2,4)$). Note that $a_{1,4}\land a_{2,3}=a_{1,3}$ and $a_{1,4}\lor a_{2,3}=a_{2,4}$.
\end{example}
\begin{example}\label{ex2}
Let $n=3$, $k=2$, $l=1$, $(i_1,i_2)=(2,3)$ and $j_1=1$. This is Possibility (2) in Proposition~\ref{diamondpairs} with $r_1=1$. The straightening relation is \[X_{2,3}X_{1}-X_{1,3}X_{2}+X_{1,2}X_{3}\] (which is the defining relation of the flag variety $F_3$). Note that $a_{2,3}\land a_1=a_{1,3}$ and $a_{2,3}\lor a_1=a_2$.
\end{example}

We will need to consider the simplest case of a \emph{classical Pl\"ucker relation}.
\begin{lemma}[{\cite[Section 9.1, Lemmas 1 and 2]{fulton}}]\label{classical}
For Pl\"ucker variables $X_{i_1,\dots,i_k}$ and $X_{j_1,\dots,j_l}$ with $k\ge l$ and any $r\in[1,l]$ we have \[X_{i_1,\dots,i_k}X_{j_1,\dots,j_l}-\sum_{s=1}^k X_{i_1,\dots,i_{s-1},j_r,i_{s+1},\dots,i_k}X_{j_1,\dots,j_{r-1},i_s,j_{r+1},\dots,j_l}\in I.\]
In the sum $j_r$ is exchanged with every $i_s$.
\end{lemma}

When $\{a,b\}$ is a diamond pair, the relation $s(a,b)$ is relatively easy to describe. \begin{proposition}\label{strrels}
For a diamond pair $\{a,b\}=\{a_{i_1,\ldots,i_k},a_{j_1,\ldots,j_l}\}$ we have $m(a,b)=1$ meaning that $s(a,b)$ contains exactly three monomials. Furthermore, one of two holds.
\begin{enumerate}
\item Possibility (1) from Proposition~\ref{diamondpairs} holds and
\begin{itemize}
\item $p_0(a,b)=a\land b=a_{\alpha_1,\ldots,\alpha_k}$ satisfies $\alpha_r=i_r$ when $r\neq r_2$ and $\alpha_{r_2}=j_{r_2}$, 
\item $q_0(a,b)=a\lor b=a_{\beta_1,\ldots,\beta_k}$ satisfies $\beta_r=j_r$ when $r\neq r_2$ and $\beta_{r_2}=i_{r_2}$,
\item $p_1(a,b)=a_{\gamma_1,\ldots,\gamma_k}$ satisfies $\gamma_r=i_r$ when $r<r_1$ or $r>r_2$, $\gamma_{r_1}=i_{r_1}$, $\gamma_{r_1+1}=j_{r_1}$ and $\gamma_r=i_{r-1}$ when $r_1+2\le r\le r_2$,
\item $q_1(a,b)=a_{\delta_1,\ldots,\delta_k}$ satisfies $\delta_r=j_r$ when $r<r_1$ or $r>r_2$, $\delta_r=j_{r+1}$ when $r_1\le r\le r_2-2$, $\delta_{r_2-1}=j_{r_2}$, $\delta_{r_2}=i_{r_2}$.
\item $c_1(a,b)=-1$.
\end{itemize}
\item Possibility (2) from Proposition~\ref{diamondpairs} holds and
\begin{itemize}
\item $p_0(a,b)=a\land b=a_{\alpha_1,\ldots,\alpha_k}$ satisfies $\alpha_r=i_r$ when $r\neq r_1$ and $\alpha_{r_1}=j_{r_1}$, 
\item $q_0(a,b)=a\lor b=a_{\beta_1,\ldots,\beta_l}$ satisfies $\beta_r=j_r$ when $r\neq r_1$ and $\beta_{r_1}=i_{r_1}$,
\item $p_1(a,b)=a_{\gamma_1,\ldots,\gamma_k}$ satisfies $\gamma_r=i_r$ when $r<r_1$, $\gamma_{r_1}=j_{r_1}$, $\gamma_{r_1+1}=i_{r_1}$ and $\gamma_r=i_{r-1}$ when $r\ge r_1+2$,
\item $q_1(a,b)=a_{\delta_1,\ldots,\delta_l}$ satisfies $\delta_r=j_r$ when $r<r_1$, $\delta_r=j_{r+1}$ when $r_1\le r\le l-1$ and $\delta_l=n$,
\item $c_1(a,b)=-1$.
\end{itemize}
\end{enumerate}
\end{proposition}
\begin{proof}
The descriptions of $a\land b$ and $a\lor b$ are immediate from the definitions. We are to show that the defined expression \[X_aX_b-X_{a\land b}X_{a\lor b}+X_{p_1(a,b)}X_{q_1(a,b)}\] lies in $I$. In fact, this expression is a classical Pl\"ucker relation given by Lemma~\ref{classical}. 

Indeed, for the first possibility the element $j_{r_2}\in\{j_1,\ldots,j_k\}$ can be exchanged with only two of the $i_s$ so that the elements in both sets stay pairwise distinct. Exchanging $j_{r_2}$ with $i_{r_2}$ turns the tuples $(i_1,\ldots,i_k)$ and $(j_1,\ldots,j_k)$ into, respectively, $(\alpha_1,\ldots,\alpha_k)$ and $(\beta_1,\ldots,\beta_k)$. Exchanging it with $i_{r_1}$ turns them into permutations of, respectively, $(\delta_1,\ldots,\delta_k)$ and $(\gamma_1,\ldots,\gamma_k)$. These two permutations are cycles of lengths $r_2-r_1$ and $r_2-r_1+1$, i.e.\ of different signs, hence $c_1(a,b)=-1$.

Similarly, for the second possibility the element $j_{r_1}$ can be exchanged with either $i_{r_1}$ or $i_k=n$.
\end{proof}

We now give the notion needed to state our main result. A \emph{special} diamond pair in $\mM$ is a diamond pair $\{a,b\}$ such that $q_1(a,b)$ covers $a\lor b$. We immediately give a reformulation.
\begin{proposition}
A diamond pair $\{a,b\}$ is special if and only if $a\land b$ covers $p_1(a,b)$.
\end{proposition}
\begin{proof}
The grading $|a|$ on $\mM$ is given by \[|a_{i_1,\ldots,i_k}|=i_1+\ldots+i_k-\frac{k(k+1)}2+\frac{(n-k)(n-k+1)}2-1.\] That is since $|a_{1,\ldots,n-1}|=0$ (minimal element of $\mM$) and for both possibilities in Proposition~\ref{cover} the grading of the covering element is greater by 1.

Now, we know that $s(a,b)$ is homogeneous with respect to $\wt$ which means that it is also homogeneous with respect to $|X_a|=|a|$. This implies that \[|q_1(a,b)|-|a\lor b|=|a\land b|-|p_1(a,b)|\] and the diamond pair is special if and only if both sides are equal to 1.
\end{proof}
A special diamond pair $\{a,b\}$ provides the following subgraph in the Hasse diagram.
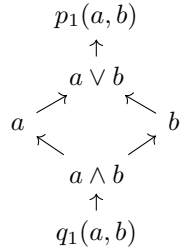
\begin{center}
\begin{tikzcd}[row sep=1mm, column sep=1mm]\label{specialpair}
&p_1(a,b)&\\[4pt]
&a\lor b\arrow[u]&\\
 a\arrow[ru]&& b\arrow[lu]\\[3pt]
 &\arrow[lu] a\land b\arrow[ru]&\\[4pt]
 &q_1(a,b)\arrow[u]&\\
\end{tikzcd}
\end{center}

An explicit description of the special diamond pairs is easy to give.
\begin{proposition}\label{specialpairs}
A diamond pair $\{a,b\}$ is special if and only if one of the two holds.
\begin{enumerate}
\item We are within possibility (1) in Proposition~\ref{diamondpairs} and $r_1=r_2-1$ and $j_{r_1}=j_{r_2}-1$ (so that $i_{r_1}$, $j_{r_1}$, $j_{r_2}$ and $i_{r_2}$ are four consecutive numbers).
\item We are within possibility (2) in Proposition~\ref{diamondpairs} and $r_1=l$ and $j_{r_1}=n-2$ (so that $j_l$, $i_l$ and $i_k$ are $n-2$, $n-1$ and $n$ respectively).
\end{enumerate}
\end{proposition}
\begin{proof}
In the first case we see from Proposition~\ref{strrels} and the above formula for $|a|$ that $|a\land b|-|p_1(a,b)|=j_{r_2}-j_{r_1}$ which implies the conditions in the proposition. In the second case we see that $|a\land b|-|p_1(a,b)|=n-i_{r_1}$ which again implies the corresponding part of the proposition.
\end{proof}

In particular, we see that the diamond pairs in both Example~\ref{ex1} and Example~\ref{ex2} are special. The following alternative way to distinguish the special diamond pairs is crucial to our proof of the main result.
\begin{lemma}\label{between}
A diamond pair $\{a,b\}$ is not special if and only if there exists an element $c$ with $p_1(a,b)<c<q_1(a,b)$ that is incomparable to at least one of $a$ and $b$ and is neither $a$ nor $b$.
\end{lemma}
\begin{proof}
Suppose that $\{a,b\}$ is not special. If we are within possibility (1) in Proposition~\ref{diamondpairs}, then $j_{r_2}-j_{r_1}>1$. Now, if $r_2-r_1=1$, we consider $c=a_{t_1,\ldots,t_k}$ with $t_r=i_r$ when $r<r_1$ or $r>r_2$, $t_{r_1}=i_{r_1}+1$ and $t_{r_2}=i_{r_2}-2$. Such a $c$ is incomparable to $a_{i_1,\ldots,i_k}$. If $r_2-r_1>1$, then we consider $c=a_{t_1,\ldots,t_k}$ with $t_r=i_r$ when $r\neq r_1+1$ and $t_{r_1+1}=i_{r_1+1}-1$. Such a $c$ is incomparable to $a_{j_1,\ldots,j_k}$. In both cases $c$ is easily seen to lie between $p_1(a,b)$ and $q_1(a,b)$.

If we are within possibility (2) in Proposition~\ref{diamondpairs}, then $i_{r_1}<n-1$. If $r_1=k-1$, we consider $c=a_{t_1,\ldots,t_k}$ with $t_r=i_r$ when $r\neq k-1$ and $t_{k-1}=i_{k-1}+1$. Such $c$ is incomparable to $a_{j_1,\ldots,j_l}$. If $r_1<k-1$, we consider $c=a_{t_1,\ldots,t_l}$ with $t_r=j_r$ when $r\neq r_1+1$ and $t_{r_1+1}=j_{r_1+1}-1$. Such a $c$ is incomparable to $a_{i_1,\ldots,i_k}$.

Now, conversely, suppose that $\{a,b\}$ is a special diamond pair and employ Proposition~\ref{specialpairs}. It is easily seen that in both cases there are only 4 elements lying strictly between $p_1(a,b)$ and $q_1(a,b)$. These are necessarily $a$, $b$, $a\land b$ and $a\lor b$ and no $c$ with the required properties exists.
\end{proof}
\begin{cor}\label{only4}
A diamond pair $\{a,b\}$ is special if and only if the only elements lying strictly between $p_1(a,b)$ and $q_1(a,b)$ are $a$, $b$, $a\land b$ and $a\lor b$.
\end{cor}

Finally, we present our first main result.
\begin{theorem}\label{maingt}
The cone $C(I,I^m)$ is composed of $w$ that satisfy inequality~\eqref{diamondineq} for every diamond pair $\{a,b\}$ and inequality 
\begin{equation}\label{specialineq}
w_a+w_b<w_{p_1(a,b)}+w_{q_1(a,b)}
\end{equation}
for every special diamond pair $\{a,b\}$. This H-description is minimal.
\end{theorem}
\begin{proof}
By Proposition~\ref{ssytredundant} we know that every facet is given by an equation of the form 
\begin{equation}\label{streq}
w_a+w_b=w_{p_i(a,b)}+w_{q_i(a,b)}
\end{equation}
for some incomparable pair $\{a,b\}$ and $i$. In view of Proposition~\ref{diamondfacets} we are to show two things. First, that for every facet given by equation~\eqref{streq} with $i\ge 1$ we, in fact, have $i=1$ and $\{a,b\}$ is a special diamond pair. Second, that for a special diamond pair $\{a,b\}$ the inequality $w_a+w_b\le w_{p_1(a,b)}+w_{q_1(a,b)}$ does not follow from all the other inequalities
\begin{equation}\label{nonstrict}
w_c+w_d\le w_{p_i(c,d)}+w_{q_i(c,d)}
\end{equation}
where either $\{c,d\}$ is a diamond pair and $i=0$ or $\{c,d\}$ is a special diamond pair and $i=1$. (Unlike the proof of Theorem~\ref{hibimonfacets}, it is more convenient for us to work with nonstrict inequalities in the second statement.)

The first claim is proved similarly to the proof of Theorem~\ref{hibimonfacets}. If $i\ge 1$ and~\eqref{streq} defines a facet, then we have a $u\in\bR^\mM$ for which $u_a+u_b>u_{p_i(a,b)}+u_{q_i(a,b)}$ but $u_c+u_d<u_{p_j(c,d)}+u_{q_j(c,d)}$ when $\{c,d\}\neq\{a,b\}$ or $j\neq i$. Suppose that $i>1$ or $\{a,b\}$ is not a special diamond pair. We then have an element $c$ with $p_i(a,b)<c<q_i(a,b)$ which is not $a$ nor $b$ and is incomparable to one of $a$ and $b$. Indeed, if $\{a,b\}$ is not a diamond pair, then either $a\lor b$ does not cover $a$ or it does not cover $b$. In the first case we may select $c$ strictly between $a$ and $a\lor b$ (then $c$ is incomparable to $b$), in the second case strictly between $b$ and $a\lor b$ (then $c$ is incomparable to $a$). If $\{a,b\}$ is a diamond pair we apply Lemma~\ref{between}. 

Denote $\deg X_aX_bX_c=\la$ and $\wt X_aX_bX_c=\mu$ and for and ideal $J\subset R$ homogeneous with respect to $\deg$ and $\wt$ denote $J_{\la,\mu}$ the corresponding homogeneous component. The dimensions of $I_{\la,\mu}$, $I^m_{\la,\mu}$ and $(\initial_u I)_{\la,\mu}$ all coincide. Now, for $X_p X_q X_r\in I^m_{\la,\mu}$ we may assume that $p$ and $q$ are incomparable. Furthermore, due to our choice of $c$ we may assume $\{p,q\}\neq\{a,b\}$, since otherwise $r=c$ which is incomparable to either $a$ or $b$ and we may rearrange $p$, $q$ and $r$. This means that $X_p X_q X_r=\initial_u (s(p,q)X_r)$. However, we see that $(\initial_u I)_{\la,\mu}$ also contains the monomial $X_{p_i(a,b)}X_cX_{q_i(a,b)}=\initial_u(s(a,b)X_c)$ which does not lie in $I^m$. This contradicts $\dim I^m_{\la,\mu}=\dim(\initial_u I)_{\la,\mu}$.

We prove the second claim by producing a point $v$ such that $v_a+v_b> v_{p_1(a,b)}+v_{q_1(a,b)}$ but\begin{equation}\label{vnonstrict}
v_c+v_d\le v_{p_i(c,d)}+v_{q_i(c,d)}
\end{equation}
whenever $\{c,d\}$ is a diamond pair and $i=0$ or $\{c,d\}\neq\{a,b\}$ is a special diamond pair and $i=1$. This point is defined as follows.
\begin{itemize}
\item Set $v_{q_1(a,b)}=0$. For all other $c$ with $|c|=|q_1(a,b)|(=|a|+2)$ set $v_c=2$.
\item Set $v_{a\lor b}=1$. For $c$ with $|c|=|a\lor b|(=|a|+1)$ such that $c<q_1(a,b)$ and $c\neq a\lor b$ set $v_c=0$. For all remaining $c$ with $|c|=|a\lor b|$ set $v_c=1$.
\item  For $c$ with $|c|=|a|$ such that $c<a\lor b$ (including $a$ and $b$) set $v_c=1$. For $c$ with $|c|=|a|$ such that $c<q_1(a,b)$ and $c\not<a\lor b$ set $v_c=0$. For all remaining $c$ with $|c|=|a|$ set $v_c=1$.
\item For all $c$ with $|c|=|a\land b|$ set $v_c=1$.
\item Set $v_{p_1(a,b)}=1$. For all other $c$ with $|c|=|p_1(a,b)|$ set $v_c=2$.
\item For all $c$ with $|c|<|p_1(a,b)|$ or $|c|>|q_1(a,b)|$ set $v_c=2^{||c|-|a||}$.
\end{itemize}
\vspace{2mm}

\hspace{-3mm}\begin{minipage}{.6\textwidth}
\setlength{\parindent}{1em}
On the right is a schematic representation of a subgraph in the Hasse diagram of $\mM$ with $v_c$ written in place of $c$. Here the value of $|c|$ ranges from $|a|-3$ to $|a|+3$, the elements $c$ in a single row have the same grading $|c|$. The red coloured elements from top to bottom are: $q_1(a,b)$, $a\lor b$, $a$ and $b$, $a\land b$, $p_1(a,b)$. The blue elements are those with $|c|=|a\lor b|$, $c<q_1(a,b)$ and $c\neq a\lor b$. The green element satisfies $|c|=|a|$, $c<a\lor b$ and $c\notin\{a,b\}$. The orange element satisfies $|c|=|a|$, $c<q_1(a,b)$ and $c\not<a\lor b$. The remaining elements are shown in black.\vspace{5mm}
\end{minipage}%
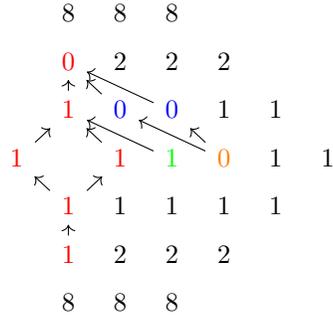
\begin{minipage}{.41\textwidth}
\raggedleft
\begin{tikzcd}[row sep=tiny, column sep=tiny]
&8&8&8\\
&\color{red} 0&2&2&2\\
&\color{red}1\arrow[u]&\color{blue}0\arrow[lu]&\color{blue}0\arrow[llu]&1&1\\
\color{red}1\arrow[ru]&& \color{red}1\arrow[lu]&\color{green}1\arrow[llu]&\color{orange}0\arrow[lu]\arrow[llu]&1&1&\\
&\color{red} 1\arrow[lu]\arrow[ru]&1&1&1&1\\
&\color{red}1\arrow[u]&2&2&2\\
&8&8&8\\ 
\end{tikzcd}
\end{minipage}

We see that $v_a+v_b> v_{p_1(a,b)}+v_{q_1(a,b)}$ and $v_a+v_b\le v_{a\land b}+v_{a\lor b}$. Consider a diamond pair $\{c,d\}$ different from $\{a,b\}$ and let $i$ be 0 or 1 with $i=1$ only if $\{c,d\}$ is special. Note that $|q_i(c,d)|-|p_i(c,d)|<=4$.

First, if $|p_i(c,d)|<|p_1(a,b)|$, then $|q_i(c,d)|<|q_1(a,b)|$ and each of  $v_c$ and $v_d$ is no greater than $v_{p_i(c,d)}/2$ and~\eqref{vnonstrict} ensues. The case of $|q_i(c,d)|>|q_1(a,b)|$ is symmetrical.

We are left to consider the case when both $|p_i(c,d)|$ and $|q_i(c,d)|$ lie in the interval $[|p_1(a,b)|,|q_1(a,b)|]$ (which contains 5 integers). Here we will make repeated use of Corollary~\ref{only4}. We have four possibilities.
\begin{enumerate}[label=\arabic*.]
\item We have $|p_i(c,d)|=|p_1(a,b)|$ and $|q_i(c,d)|=|q_1(a,b)|$, i.e.\ $\{c,d\}$ is a special diamond pair with $|c|=|d|=|a|=|b|$ and $i=1$. At least one of $c$ and $d$ does not lie between $p_1(a,b)$ and $q_1(a,b)$, therefore $p_i(c,d)\neq p_1(a,b)$ or $q_i(c,d)\neq q_1(a,b)$ which means that at least one of $v_{p_i(c,d)}$ and $v_{q_i(c,d)}$ is equal to 2. Meanwhile, both $v_c$ and $v_d$ are 0 or 1, \eqref{vnonstrict} follows.
\item We have $|q_i(c,d)|=|q_1(a,b)|$ and $|p_i(c,d)|=|q_1(a,b)|-2$, i.e.\ $\{c,d\}$ is a diamond pair with $|c|=|d|=|a|+1$ and $i=0$. If $q_i(c,d)\neq q_1(a,b)$, then \eqref{vnonstrict} is immediate, since $v_{q_i(c,d)}=2$ while all other terms are 0 or 1. If $q_i(c,d)=q_1(a,b)$ and one of $c$ and $d$ is $a\lor b$, then both sides of~\eqref{vnonstrict} are $0+1$. If $q_i(c,d)=q_1(a,b)$ and neither of $c$ and $d$ is $a\lor b$, then $v_c=v_d=0$ and~\eqref{vnonstrict} follows. 
\item We have $|q_i(c,d)|=|a\lor b|$ and $|p_i(c,d)|=|a\land b|$, i.e.\ $\{c,d\}$ is a diamond pair with $|c|=|d|=|a|$ and $i=0$. Since $v_{p_i(c,d)}=1$ and all other terms in~\eqref{vnonstrict} are 0 or 1, we only discuss the case when $v_{q_i(c,d)}=0$, i.e.\ $q_1(a,b)>q_i(c,d)\neq a\lor b$. We must show that we cannot have both $v_c=1$ and $v_d=1$, which would mean that $c<a\lor b$ and $d<a\lor b$. But we must then have $c\lor d=a\lor b$ which contradicts $q_i(c,d)\neq a\lor b$.
\item We have $|q_i(c,d)|=|a|$ and $|p_i(c,d)|=|p_1(a,b)|$, i.e.\ $\{c,d\}$ is a diamond pair with $|c|=|d|=|a\land b|$ and $i=0$. If $p_i(c,d)\neq p_1(a,b)$, then $v_{p_i(c,d)}=2$ and~\eqref{vnonstrict} is clear. If $p_i(c,d)=p_1(a,b)$, we must show that $v_{q_i(c,d)}\neq 0$, i.e.\ that $q_i(c,d)=c\lor d$ may not be less than $q_1(a,b)$ but not less than $a\lor b$. However, if $q_i(c,d)<q_1(a,b)$, then $q_i(c,d)$ lies strictly between $p_i(c,d)=p_1(a,b)$ and $q_1(a,b)$ which implies that $q_i(c,d)$ is one of $a$ and $b$.
\end{enumerate}
\end{proof}

We conclude this section by pointing out that the number of facets of $\overline{C(I,I^m)}$ is easy to compute and the corresponding sequence (multiplied by $\frac12$) can be found in the Online Encyclopedia of Integer Sequences~\cite{OEIS} where other interpretations are given.

\begin{cor}
The number of facets of $\overline{C(I,I^m)}$ is equal to $2^{n-5}(n^2+n-4)$ (this is, modulo a numeration shift, sequences \href{http://oeis.org/A001793}{A049611} and \href{http://oeis.org/A084851}{A084851} in OEIS multiplied by 2). 
\end{cor}
\begin{proof}
The number of facets is equal to the number of diamond pairs in $\mM$ plus the number of special diamond pairs in $\mM$. To obtain a diamond pair of the type given by possibility (1) in Proposition~\ref{diamondpairs} one needs to choose the integers $1<i_{r_1}<i_{r_2}\le n$ with $i_{r_2}-i_{r_1}\ge 3$ (which can be done in $\frac{n^2-5n+6}2$ ways) and then choose the remaining $i_r$ from $[1,n]\backslash\{i_{r_1},i_{r_1}+1,i_{r_2}-1,i_{r_2}\}$ which can be done in $2^{n-4}$ ways. Similarly, for possibility (2) one obtains $2^{n-3}(n-2)$ diamond pairs which gives us a total of $2^{n-5}(n^2-n-2)$ diamond pairs.

Possibility (1) in Proposition~\ref{specialpairs} provides $2^{n-4}(n-3)$ special diamond pairs and possibility (2) provides $2^{n-3}$. The proposition follows.
\end{proof}
In particular, we obtain
\begin{cor}
The number of facets of $\overline{C(I^h,I^m)}$ is equal to $2^{n-5}(n^2-n-2)$ (this is, modulo a numeration shift, sequence \href{http://oeis.org/A001793}{A001793} in OEIS). 
\end{cor}

\section{Interpolating polytopes}

The goal of the next two sections is to define a family of toric ideals generalizing the Hibi ideal for any finite distributive lattice and to describe certain maximal cones in their Gr\"obner fans. This family of ideals is parametrized by a family of polytopes which we study in this section.

Consider an arbitrary finite poset $(P,<)$ and let $U_o, U_c\subset P$ form a partition: $U_o\cup U_c=P$ and $U_o\cap U_c=\varnothing$. Here ``$o$'' stands for ``order'' and ``$c$'' for ``chain'', the reasons will become clear below. We define a polytope $\Pi_{U_o,U_c}(P)\subset\bR^P$ consisting of points $x$ with coordinates $\{x_p,p\in P\}$ such that:
\begin{enumerate}
\item for any $p\in P$ one has $0\le x_p\le 1$,
\item for any $p,q\in U_o$ with $p<q$ one has $x_q\le x_p$ (note the order!),
\item for any $p_1<\ldots<p_k$ with $p_i\in U_c$ for $i<k$ and $p_k\in P$ one has $\sum x_{p_i}\le 1$,
\item for any $q<p_1<\ldots<p_k$ with $q\in U_o$, $p_i\in U_c$ for $i<k$ and $p_k\in P$ one has $\sum x_{p_i}\le x_q$.
\end{enumerate}

The polytope $\Pi_{P,\varnothing}(P)$ consists of points satisfying $0\le x_q\le x_p\le 1$ for any $p<q$. This is commonly referred to as the \emph{order polytope} of $P$, but note that the original definition in~\cite{stan} requires $x_p\le x_q$ instead which amounts to reflecting the polytope in the point $(\frac12,\ldots,\frac12)$. The polytope $\Pi_{\varnothing,P}(P)$ consists of points satisfying $x_p\ge 0$ for all $p$ and $\sum x_{p_i}\le 1$ for any chain $(p_1,\ldots,p_k)$. This is the \emph{chain polytope} of $P$, also defined in~\cite{stan}. In general, we will refer to the polytopes $\Pi_{U_o,U_c}(P)$ as \emph{interpolating polytopes}, since they can be said to interpolate between the order and chain polytope. They are similar to but different from the \emph{order-chain polytopes} of~\cite{HLLLT}, \emph{marked chain-order polytopes} of~\cite{FF} and the \emph{marked poset polytopes} of~\cite{FFLP}, which are three other families interpolating between the order and chain polytopes. 
\begin{remark}
To make the above more precise, the polytope $\Pi_{U_o,U_c}(P)$ can, in fact, be represented as a unimodular transform of a certain marked poset polytope from~\cite{FFLP}. To see this consider the poset $Q=P\cup\{\varepsilon_0,\varepsilon_1\}$ such that $\varepsilon_0<p<\varepsilon_1$ for any $p\in P$. Then, in the terminology of~\cite[Subsection 1.3]{FFLP}, let $\{\varepsilon_0,\varepsilon_1\}$ be the set of marked elements in $Q$, let $U_o$ be the set of order elements, $U_c$ be the set of chain elements and consider the marking $\lambda:\varepsilon_i\mapsto i$. We obtain the \emph{marked chain-order polyhedron} (a subclass of marked poset polytopes) $\mathcal O_{U_c,U_o}(Q,\lambda)\subset\bR^Q$. One may then verify that $\Pi_{U_o,U_c}(P)$ is obtained from $\mathcal O_{U_c,U_o}(Q,\lambda)$ by projection onto $\bR^P$ followed by reflection in the point $x$ with $x_p=\frac12$ for $p\in U_o$ and $x_p=0$ otherwise. Subsequently, two of the facts proved in this section (Lemma~\ref{ehrhart} and Proposition~\ref{minkowski}) follow from results in~\cite{FFLP} and~\cite{FFP} and independent proofs are given for completeness only.
\end{remark}

\begin{example}\label{polytopeexample}
Consider the poset $(P=\{p,q,r,s\},<)$ with covering relations $p<q$, $q<r$ and $q<s$ so that its Hasse diagram is as follows:
\begin{center}
\begin{tikzcd}[row sep=1mm,column sep=tiny]
&&r\\
p\arrow[r]&q\arrow[ru]\arrow[rd]&\\
&&s\\
\end{tikzcd}
\end{center}
One easily checks that 
\begin{itemize}
\item the order polytope $\mathcal O_{P,\varnothing}$ is given by the inequalities $1\ge x_p\ge x_q\ge x_r\ge 0$ and $x_q\ge x_s\ge0$,
\item the chain polytope $\mathcal O_{\varnothing,P}$ is given by the inequalities $x_p,x_q,x_r,x_s\ge 0$, $x_p+x_q+x_r\le1$ and $x_p+x_q+x_s\le1$,
\item the interpolating polytope $\mathcal O_{\{p\},\{q,r,s\}}$ is given by the inequalities $1\ge x_p\ge x_q+x_r$, $x_p\ge x_q+x_s$ and $x_q,x_r,x_s\ge 0$. 
\end{itemize}
In fact, the inequalities above provide the facets of the polytopes while all other inequalities given by conditions (1)-(4) follow.
\end{example}

\begin{lemma}[{cf.~\cite[Corollary 2.5]{FFLP}}]\label{ehrhart}
Any interpolating polytope is Ehrhart equivalent to the order polytope: for any integer $t\ge 0$ the $t$-dilation $t\Pi_{U_o,U_c}(P)$ contains the same number of integer points as $t\Pi_{P,\varnothing}(P)$. 
\end{lemma}
\begin{proof}
We prove the lemma by defining a piecewise linear bijection between the two dilations that preserves integer points. This bijection is similar to the bijection in~\cite{stan} between the order and chain polytopes.

Define a map $\zeta:\bR^P\to \bR^P$ by setting $\zeta(x)_p=x_p$ if $p\in U_o$ or $p$ is maximal in $P$ and setting \[\zeta(x)_p=\min_{q\in P,q>p} (x_p-x_q)\] for all other $p\in U_c$. We claim that $\zeta(t\Pi_{P,\varnothing}(P))\subset t\Pi_{U_o,U_c}(P)$. Consider $x\in t\Pi_{P,\varnothing}(P)$, then it is obvious that $\zeta(x)$ satisfies (the $t$-dilations of) conditions (1) and (2) in the definition of the interpolating polytope. Now consider $p_1<\ldots<p_k$ with $p_i\in U_c$ for $i<k$ and $p_k\in P$. We see that \[\sum\zeta(x)_{p_i}\le(x_{p_1}-x_{p_2})+\dots+(x_{p_{k-1}}-x_{p_k})+x_{p_k}= x_{p_1}\le t.\] Consider $q<p_1<\ldots<p_k$ with $q\in U_o$, $p_i\in U_c$ for $i<k$ and $p_k\in P$, similarly, \[\sum\zeta(x)_{p_i}\le x_{p_1}\le x_q=\zeta(x)_q.\]

Now define a map $\zeta':\bR^P\to \bR^P$ by setting $\zeta'(x)_p=x_p$ if $p\in U_o$ and \[\zeta'(x)_p=\max_{p=p_1<\ldots<p_k}\sum_{i=1}^k x_{p_i}\] for $p\in U_c$ where the sum is taken over all chains with $p_1=p$, $p_i\in U_c$ for $i<k$ and $p_k\in P$. Consider $x\in t\Pi_{U_o,U_c}(P)$, we show that $\zeta'(x)\in t\Pi_{P,\varnothing}(P)$. We have $0\le \zeta'(x)_p\le t$ for every $p$ by (the $t$-dilations of) conditions (1) and (3) in the definition of $\Pi_{U_o,U_c}(P)$. Consider a pair $p<q$. If either $p,q\in U_o$ or $p,q\in U_c$, then $\zeta'(x)_p\ge \zeta'(x)_q$ is immediate. If $p\in U_o$ and $q\in U_c$, then $\zeta'(x)_q\le x_p=\zeta'(x)_p$ by condition (4). If $p\in U_c$ and $q\in U_o$, then obviously $\zeta'(x)_p\ge x_q=\zeta'(x)_q$.

To complete the proof we show that $\zeta(\zeta'(x))=x$ for $x\in t\Pi_{U_o,U_c}(P)$ and $\zeta'(\zeta(x))=x$ for $x\in t\Pi_{P,\varnothing}(P)$. Indeed, for any $x\in\bR^P$ for $p\in U_o$ or $p$ maximal in $P$ we simply have $\zeta(x)_p=\zeta'(x)_p=x_p$. Consider any other $p\in U_c$ and some $x\in t\Pi_{U_o,U_c}(P)$. Then $\zeta'(x)_p=x_p+\max_{q>p}\zeta'(x)_q$ which implies $\zeta(\zeta'(x))_p=x_p$. Now consider $x\in t\Pi_{P,\varnothing}(P)$. Set $p_1=p$, choose $p_2>p_1$ with the largest $x_{p_2}$, then $p_3>p_2$ with the largest $x_{p_3}$ and so on until we reach a $p_k$ that is either maximal in $P$ or lies in $U_o$. Evidently, $\zeta(x)_{p_i}=x_{p_i}-x_{p_{i+1}}$ for $i<k$ and $\zeta(x)_{p_k}=x_{p_k}$, consequently, $\sum \zeta(x)_{p_i}=x_p$ while for any other chain $p=q_1<\ldots<q_l$ with $q_i\in U_c$ for $i<k$ and $q_k\in P$ we have $\sum \zeta(x)_{q_i}\le x_p$. We obtain $\zeta'(\zeta(x))_p=x_p$.  
\end{proof}

In fact, the bijection $\zeta$ has a simple combinatorial description when $t=1$. First, note that the set of integer points in the order polytope $\Pi_{P,\varnothing}(P)$ consists of the indicator functions $\mathbf 1_J\in\bR^P$ of order ideals $J\subset P$. (For us $J$ is an order ideal if $p\in J$ and $q<p$ implies $q\in J$, i.e.\ it is a lower set.) Let $\zeta$ be as in the above proof. For an order ideal $J\subset P$ let $K_{U_o,U_c}(J)$ denote the set of such $p$ that $\zeta(\mathbf 1_J)_p=1$ (while $\zeta(\mathbf 1_J)_p=0$ for all other $p$). The following is immediate from the definition of $\zeta$.
\begin{proposition}\label{combdesc}
For an order ideal $J\subset P$ the set $K_{U_o,U_c}(J)$ consists of all $p\in J\cap U_o$ and those $p\in J\cap U_c$ that are maximal in $J$.
\end{proposition}

In particular, $K_{P,\varnothing}(J)=J$ and $K_{\varnothing,P}(J)$ is the antichain composed of the maximal elements in $J$. 
\begin{example}\label{polytopeexample1}
Let us consider the polytopes from Example~\ref{polytopeexample}. One sees that in the notation $(x_p,x_q,x_r,x_s)$
\begin{itemize}
\item the integer points in $\mathcal O_{P,\varnothing}$ are $(0,0,0,0)$, $(1,0,0,0)$, $(1,1,0,0)$, $(1,1,1,0)$, $(1,1,0,1)$, $(1,1,1,1)$,
\item the integer points in $\mathcal O_{\varnothing,P}$ are $(0,0,0,0)$, $(1,0,0,0)$, $(0,1,0,0)$, $(0,0,1,0)$, $(0,0,0,1)$, $(0,0,1,1)$,
\item the integer points in $\mathcal O_{\{p\},\{q,r,s\}}$ are $(0,0,0,0)$, $(1,0,0,0)$, $(1,1,0,0)$, $(1,0,1,0)$, $(1,0,0,1)$, $(1,0,1,1)$.
\end{itemize}
All three polytopes have the same number of integer points in accordance with Lemma~\ref{ehrhart}. Moreover, these integer points indeed have the forms $\mathbf 1_{K_{U_o,U_c}(J)}$ where $J$ is one of the order ideals $\varnothing,\{p\},\{p,q\},\{p,q,r\},\{p,q,s\},\{p,q,r,s\}$ and $K_{U_o,U_c}(J)$ is described by Proposition~\ref{combdesc}.
\end{example}

The following properties of this correspondence will be important to us.
\begin{proposition}\label{union}
For order ideals $J_1,J_2\subset P$ one has \[K_{U_o,U_c}(J_1)\cap K_{U_o,U_c}(J_2)\subset K_{U_o,U_c}(J_1\cup J_2) \subset K_{U_o,U_c}(J_1)\cup K_{U_o,U_c}(J_2).\] 
\end{proposition}
\begin{proof}
This is easily derived from Proposition~\ref{combdesc}. For the first inclusion note that \[K_{U_o,U_c}(J_1)\cap K_{U_o,U_c}(J_2)\cap U_o=J_1\cap J_2\cap U_o\subset (J_1\cup J_2)\cap U_o=K_{U_o,U_c}(J_1\cup J_2)\cap U_o.\] If $p\in K_{U_o,U_c}(J_1)\cap K_{U_o,U_c}(J_2)\cap U_c$, then $p$ is maximal in both $J_1$ and $J_2$, hence it is maximal in $J_1\cup J_2$ and lies in $K_{U_o,U_c}(J_1\cup J_2)$.

For the second inclusion, \[K_{U_o,U_c}(J_1\cup J_2)\cap U_o=(J_1\cup J_2)\cap U_o= (K_{U_o,U_c}(J_1)\cup K_{U_o,U_c}(J_2))\cap U_o.\] If $p\in U_c$ is maximal in $J_1\cup J_2$, then $p$ is maximal in the $J_i$ that contains it and, therefore, $p$ lies in $K_{U_o,U_c}(J_1)\cup K_{U_o,U_c}(J_2)$.
\end{proof}

The above proposition implies that for any order ideals $J_1,J_2\subset P$ there exists a subset $D\subset P$ such that for indicator functions we have 
\begin{equation}\label{indicator}
\mathbf 1_{K_{U_o,U_c}(J_1)}+\mathbf 1_{K_{U_o,U_c}(J_2)}=\mathbf 1_{K_{U_o,U_c}(J_1\cup J_2)}+\mathbf 1_D.
\end{equation}

\begin{proposition}\label{star}
In the above notations $D\subset K_{U_o,U_c}(J_1\cap J_2)$ and there exists an order ideal $J'\subset P$ such that $D=K_{U_o,U_c}(J')$.
\end{proposition}
\begin{proof}
One easily sees that $D$ is the disjoint union of $K_{U_o,U_c}(J_1)\cap K_{U_o,U_c}(J_2)$ and $(K_{U_o,U_c}(J_1)\cup K_{U_o,U_c}(J_2))\backslash K_{U_o,U_c}(J_1\cup J_2)$. 

Now, for the first claim it is evident that both $D\cap U_o$ and $K_{U_o,U_c}(J_1\cap J_2)\cap U_o$ are equal to $J_1\cap J_2\cap U_o$. Consider a $p\in D\cap U_c$, there are two possibilities. Either $p\in K_{U_o,U_c}(J_1)\cap K_{U_o,U_c}(J_2)$ which means that $p$ is maximal in both of $J_1$ and $J_2$ and is, therefore, maximal in $J_1\cap J_2$. Or $p$ is maximal in just one of the $J_i$, say $J_1$, and is not contained in $K_{U_o,U_c}(J_1\cup J_2)$, i.e.\ there exists a $q\in J_2$ with $q>p$. This already means that $p$ lies in $J_1\cap J_2$ and is maximal therein (since it is maximal in $J_1$).

For the second claim, $J'$ is chosen as the minimal ideal containing $D$. First we show that $J'\cap U_o=J_1\cap J_2\cap U_o$. For this we must show that there are no $p\in D\cap U_c$ and $q\in U_o\backslash(J_1\cap J_2)$ with $p>q$. Indeed, such a $p$ would lie in $J_1\cap J_2$ by the first claim of the proposition and $p>q$ would then imply $q\in J_1\cap J_2\cap U_o$. We obtain $J'\cap U_o=D\cap U_o$. 

By Proposition~\ref{combdesc} we are left to show that $D\cap U_c$ is the set of maximal elements in $J'$ that lie in $U_c$. Note that the sets of maximal elements in $J'$ and $D$ coincide. Hence we only need to show that any $p\in D\cap U_c$ is maximal in $D$. If $p\in U_c\cap K_{U_o,U_c}(J_1)\cap K_{U_o,U_c}(J_2)$, then $p$ is maximal in both $J_1$ and $J_2$ and thus in $D$. Suppose that $p\in U_c\cap(K_{U_o,U_c}(J_1)\cup K_{U_o,U_c}(J_2))\backslash K_{U_o,U_c}(J_1\cup J_2)$, then it is maximal in one of the $J_i$, say $J_1$, but not the other. If we have $q\in D$ with $q>p$, then $q\in J_2$ and we may assume $q$ is maximal in $J_2$. This means that $q$ is maximal in $J_1\cup J_2$ but $q\notin J_1$ which contradicts $q\in D$.
\end{proof}

In the notations of Proposition~\ref{star}, we will write $J_1\odot_{U_o,U_c} J_2$ to denote the order ideal $J'$. We see that $J_1\odot_{U_o,U_c} J_2\subset J_1\cap J_2$ and that $J_1\odot_{P,\varnothing} J_2= J_1\cap J_2$. The operation $\odot_{\varnothing,P}$ is the operation $*$ considered in~\cite{HL} where $J_1*J_2$ is the ideal generated by $\max(J_1\cap J_2)\cap(\max(J_1)\cup\max(J_2))$. One also sees that when $J_1\subset J_2$ we have $J_1\odot_{U_o,U_c} J_2=J_1$.

\begin{example}
The poset $(P,<)$ from Example~\ref{polytopeexample} only has one pair of ideals for which $J_1\not\subset J_2$ and $J_2\not\subset J_1$, that is $J_1=\{p,q,r\}$ and $J_2=\{p,q,s\}$. With the use of Example~\ref{polytopeexample1} one checks that $J_1\odot_{P,\varnothing}J_2=\{p,q\}$, $J_1\odot_{\varnothing,P}J_2=\varnothing$ and $J_1\odot_{\{p\},\{q,r,s\}}J_2=\{p\}$.
\end{example}

We will need one more property of interpolating polytopes, namely the so-called \emph{Minkowski sum property}.
\begin{proposition}[{cf.~\cite[Theorem 20]{FFP}}]\label{minkowski}
For any integer $t\ge 1$ and an integer point $x\in t\Pi_{U_o,U_c}(P)$ there exist integer points $x_i\in \Pi_{U_o,U_c}(P)$ with $1\le i\le t$ such that $x=x_1+\ldots+x_t$.
\end{proposition}
\begin{proof}
First, suppose that $U_o=P$, consider $i\in[1,t]$, then the set of $p$ with $x_p\ge i$ forms an order ideal $J_i\in P$. Setting $x_i=\mathbf 1_{J_i}$ proves the result in this case. We have also obtained a bijection between the integer points in $t\Pi_{P,\varnothing}(P)$ an (weakly) decreasing chains of order ideals $J_1\supset\ldots\supset J_t$.

Consider the general case. By Lemma~\ref{ehrhart} it now suffices to show that the sums \[S=\mathbf 1_{K_{U_o,U_c}(J_1)}+\dots+\mathbf 1_{K_{U_o,U_c}(J_t)}\] given by all decreasing chains of order ideals $J_1\supset\ldots\supset J_t$ are pairwise distinct, since then there will be no other integer points in $t\Pi_{U_o,U_c}(P)$. We prove this by induction on $t$, the base $t=1$ being trivial. For the induction step it suffices to show that the order ideal $J_1$ is uniquely determined by the sum $S$. Indeed, one easily sees that the set of maximal elements in $J_1$ must coincide with the set of maximal elements in $\{p|S_p>0\}$.
\end{proof}

\section{Generalized Hibi ideals}

We will now apply the notions introduced in the previous section to the study of distributive lattices. For this we will need the fundamental theorem of finite distributive lattices.

Let $\mL$ be a finite distributive lattice. An element $a\in\mL$ is called join-irreducible if $a$ is not minimal in $\mL$ and $a=b\lor c$ implies $a=b$ or $a=c$. This is equivalent to $a$ covering exactly one element. Let $\mathcal P(\mL)$ be the set of join-irreducible elements in $\mL$. We view $\mathcal P(\mL)$ as a poset with the induced relation $<$. For any poset $P$ let $\mathcal J(P)$ be the set of order ideals in $P$. It is easy to see that $\mathcal J(P)$ is a distributive lattice with union (of order ideals) as join, intersection as meet and inclusion as the order relation. The following classical result due to Garret Birkhoff is known as \emph{Birkhoff's representation theorem} or the \emph{fundamental theorem of finite distributive lattices}, a proof can be found in~\cite[Theorem 9.1.7]{HH1}.
\begin{theorem}
The distributive lattices $\mL$ and $\mathcal J(\mathcal P(\mL))$ are isomorphic. An isomorphism $\iota_\mL$ is obtained by mapping $a\in\mL$ to the order ideal in $\mathcal P(\mL)$ composed of all join-irreducible elements $p$ with $p\le a$.  
\end{theorem}

In particular, one sees that $a$ covers $b$ if and only if $\iota_\mL(a)$ is obtained from $\iota_\mL(b)$ by adding one element. Therefore, $|a|$ is equal to the cardinality of the order ideal $\iota_\mL(a)$. 


Before defining the family of ideals that generalizes the Hibi ideal let us explain an alternative approach to the Hibi ideal via the fundamental theorem. Consider variables $z_p$ indexed by $p\in\mathcal P(\mL)$ and let $S=\bC[\{z_p\},t]$. Define a homomorphism $\theta:R(\mL)\to S$ given by $\theta(X_a)=t\prod_{p\in\iota_\mL(a)}z_p$. It is immediate from the above that the generators of $I^h(\mL)$ lie in the kernel of $\theta$. Moreover, the following was observed in~\cite[Section 2]{H}.
\begin{proposition}
$I^h(\mL)$ is the kernel of $\theta$.
\end{proposition}

Now consider a partition of $U_o\cup U_c=\mathcal P(\mL)$. Define a homomorphism $\theta_{U_o,U_c}:R(\mL)\to S$ given by \[\theta_{U_o,U_c}(X_a)=t\prod_{p\in K_{U_o,U_c}(\iota_\mL(a))}z_p.\] Let the \emph{generalized Hibi ideal} $I^{U_o,U_c}(\mL)$ be the kernel of this map. We see that $I^{P,\varnothing}(\mL)=I^h(\mL)$. The ideal $I^{\varnothing,P}(\mL)$ has been studied in~\cite{HL}.

For $a,b\in\mL$ let us use the notation \[a\odot_{U_o,U_c}b=\iota_\mL^{-1}(\iota_\mL(a)\odot_{U_o,U_c}\iota_\mL(b)).\] Note that $a\odot_{U_o,U_c}b\le a\land b$.
\begin{proposition}
There exists a $w\in\bR^\mL$ such that $\initial_w I^{U_o,U_c}(\mL)=I^m(\mL)$.
\end{proposition}
\begin{proof}
Set $w_a=|a|^2$. For incomparable $a$ and $b$ consider the expression \[d_{U_o,U_c}(a,b)=X_aX_b-X_{a\lor b}X_{a\odot_{U_o,U_c}b}.\] This expression lies in $I^{U_o,U_c}(\mL)$ by the definition of $\odot_{U_o,U_c}$. We have $|a|+|b|=|a\land b|+|a\lor b|$ and $a\odot_{U_o,U_c}b\le a\land b$ which means that $\initial_w d_{U_o,U_c}(a,b)=X_a X_b$. We see that $I^m(\mL)\subset \initial_w I^{U_o,U_c}(\mL)$, we are left to show that the corresponding homogeneous components of these two ideals have the same dimensions. However, Proposition~\ref{minkowski} shows that for any homogeneity degree $d$ the dimension of $\theta_{U_o,U_c}(R_d)$ is equal to the number of integer points in $d\Pi_{U_o,U_c}(\mathcal P(\mL))$. From Lemma~\ref{ehrhart} we now see that this dimension is independent of the partition and is equal to \[\dim\theta_{\mathcal P(\mL),\varnothing}(R_d)=\dim R_d-\dim I^h(\mL)_d=\dim R^d-\dim I^m(\mL)_d.\qedhere\]
\end{proof}

As a corollary of the proof we obtain
\begin{cor}\label{quadgens}
The ideal $I^{U_o,U_c}(\mL)$ is generated by the expressions $d_{U_o,U_c}(a,b)$ with $\{a,b\}$ ranging over incomparable pairs in $\mL$.
\end{cor}
\begin{proof}
We see that for the ideal $J$ generated by the $d_{U_o,U_c}(a,b)$ we have $I^m(\mL)\subset \initial_w J$. Hence $J=I^{U_o,U_c}(\mL)$. 
\end{proof}
\begin{remark}\label{generalizedtoric}
The above results show that the variety cut out by $I^{U_o,U_c}(\mL)$ in $\bP(\bC^\mL)$ is the toric variety associated with the polytope $\Pi_{U_o,U_c}(\mathcal P(\mL))$.
\end{remark}

Proposition~\ref{redundant} and Theorem~\ref{hibimonfacets} also generalize.
\begin{proposition}
$C(I^{U_o,U_c}(\mL),I^m(\mL))$ is nonempty and consists of the points $w$ satisfying
\begin{equation}\label{hibiliineq}
w_a+w_b<w_{a\lor b}+w_{a\odot_{U_o,U_c}b}
\end{equation}
for all incomparable $a$ and $b$. 
\end{proposition}
\begin{proof}
The proof is completely analogous to that of Proposition~\ref{redundant}.
\end{proof}

\begin{theorem}\label{genhibimonfacets}
$C(I^{U_o,U_c}(\mL),I^m(\mL))$ consists of the points $w$ satisfying~\eqref{hibiliineq} for all diamond pairs $\{a,b\}$ in $\mL$. This H-description is minimal.
\end{theorem}
\begin{proof}
The proof is rather similar to that of Theorem~\ref{hibimonfacets}. Choose a facet of $\overline{C(I^{U_o,U_c}(\mL),I^m(\mL))}$, we may assume that it is contained in the hyperplane given by $w_a+w_b=w_{a\lor b}+w_{a\odot_{U_o,U_c}b}$ for some incomparable $a$ and $b$. We may choose a point $u$ such that $u_c+u_d<u_{c\lor d}+u_{c\odot_{U_o,U_c}d}$ for any incomparable pair $\{c,d\}$ other than $\{a,b\}$ but $u_a+u_b>u_{a\lor b}+u_{a\odot_{U_o,U_c}b}$.

Furthermore, if $\{a,b\}$ is not a diamond pair, we may choose an element $c$ which is different from $a$ and $b$ and satisfies $a\land b<c<a\lor b$. We achieve a contradiction by showing that $\dim(\initial_u I^{U_o,U_c}(\mL))_3\neq\dim I^m(\mL)_3$. For $M=X_pX_qX_r\in I^m(\mL)$ we define $\varphi(M)=M$ if $\{p,q,r\}$ contains an incomparable pair different from $\{a,b\}$. For $M=X_aX_bX_r$ with $r\le a\land b$ or $r\ge a\lor b$ we set $\varphi(M)=X_{a\odot_{U_o,U_c}b}X_{a\lor b}X_r$. We see that $\varphi$ is an injective map from the set of monomials in $I^m(\mL)_3$ to the set of monomials in $(\initial_u I^{U_o,U_c}(\mL))_3$. However, the image of $\varphi$ does not contain the monomial $X_{a\odot_{U_o,U_c}b}X_{a\lor b}X_c$ which is contained in $(\initial_u I^{U_o,U_c}(\mL))_3$. 

Now, for a diamond pair $\{a,b\}$ we are left to find a point $v$ such that $v_a+v_b\ge v_{a\lor b}+v_{a\odot_{U_o,U_c}b}$ but $v_c+v_d<v_{c\lor d}+v_{c\odot_{U_o,U_c}d}$ for any other diamond pair $\{c,d\}$. We set $A=3^{|a|-|a\odot_{U_o,U_c}b|}$ and define 
\begin{equation*}
v_c=
\begin{cases}
A,\text{ if }c\in\{a,b\},\\
A^{|c|-|a|}\text{ if }|c|\ge|a|\text{ and }c\notin\{a,b\},\\
v_c=A^{\frac{|a|-|c|}{|a|-|a\odot_{U_o,U_c}b|}}=3^{|a|-|c|}\text{ if }|c|<|a|.
\end{cases}
\end{equation*}
We see that $v_a+v_b= v_{a\lor b}+v_{a\odot_{U_o,U_c}b}=A+A$. Consider a diamond pair $\{c,d\}\neq\{a,b\}$. We see that if $|c|\ge|a|$ and neither $c$ nor $d$ lie in $\{a,b\}$, then $v_{c\lor d}>v_c+v_d$. If $|c|<|a|$, then $v_{c\odot_{U_o,U_c} d}>v_c+v_d$. If one of $c$ and $d$ lies in $\{a,b\}$, then $v_c+v_d=A+1$ while $v_{c\lor d}=A$ and $v_{c\odot_{U_o,U_c}d}>1$. 
\end{proof}

\section{PBW-semistandard tableaux}

The goal of this section is to show that the distributive lattices corresponding to semistandard and PBW-semistandard Young tableaux are isomorphic and to deduce that the toric initial ideal corresponding to the FFLV polytope of~\cite{FFL1} is, in fact, a generalized Hibi ideal. We choose an $n\ge2$ and let lattice $\mM$ be as in Section~\ref{semistandard}. First of all, let us see how the fundamental theorem applies to $\mM$. 
\begin{proposition}\label{irreducible}
An element $a_{i_1,\ldots,i_k}$ is join-irreducible if and only if the set $\{i_1,\ldots,i_k\}$ is of the form $[1,n]\backslash[r,s]$ where $1\le r\le s\le n$ and $[r,s]$ is not $[1,n]$ or $[n,n]$. For two such sets $\{i_1,\ldots,i_k\}=[1,n]\backslash[r,s]$ and $\{j_1,\ldots,j_l\}=[1,n]\backslash[u,v]$ one has $a_{i_1,\ldots,i_k}\le a_{j_1,\ldots,j_l}$ if and only if $r\ge u$ and $s-r\le v-u$.
\end{proposition}
\begin{proof}
The element $a_{i_1,\ldots,i_k}$ is join-irreducible if and only if it covers exactly one element. By Proposition~\ref{cover} this means that either there is exactly one $t\in[1,k]$ such that $i_t>i_{t-1}+1$ (where $i_0=0$) and $i_k=n$ or there are no such $t$, $i_k<n$ and $k<n-1$. The first claim follows. To obtain the second claim first note that $s-r\le v-u$ is equivalent to $k\ge l$. Now, for $t\ge r$ we have $i_t=t-r+s+1$ and for $t\ge u$ we have $j_t=t-u+v+1$. Therefore, if $k\ge l$ and $r<u$ we have $i_r>j_r$ while if $k\ge l$ and $r\ge u$ we have all $i_t\le j_t$.
\end{proof}

For our purposes it will be more convenient to numerate the join-irreducible elements somewhat differently. For $1\le r\le s\le n$ and $(r,s)$ not one of $(1,1)$ and $(n,n)$ let $y_{r,s}\in\mathcal P(\mM)\subset\mM$ be $a_{i_1,\ldots,i_k}$ with $\{i_1,\ldots,i_k\}=[1,n]\backslash[n-s+1,n+r-s]$. Now we have $y_{r,s}\le y_{u,v}$ if and only if $r\le u$ and $s\le v$. The poset $\mathcal P(\mM)$ is easy to visualize, for instance, when $n=4$ its Hasse diagram looks like this (cf. also Example~\ref{hasse}):
\begin{equation}
\begin{tikzcd}[row sep=1mm,column sep=tiny]\label{n=4}
&y_{2,2}\arrow[rd]&&y_{3,3}\arrow[rd]\\
y_{1,2}\arrow[rd]\arrow[ru]&&y_{2,3}\arrow[rd]\arrow[ru]&&y_{3,4}\\
&y_{1,3}\arrow[rd]\arrow[ru]&&y_{2,4}\arrow[ru]\\
&&y_{1,4}\arrow[ru]
\end{tikzcd}
\end{equation}
In particular, note that for any incomparable $y_{r,s}$ and $y_{u,v}$ we may assume that $u<r\le s<v$ and, more generally, for any antichain $\{y_{r_1,s_1},\dots,y_{r_m,s_m}\}$ we may assume that $r_m<\dots<r_1\le s_1<\dots<s_m$.

We now consider the interpolating polytope $\Pi_{U_o,U_c}(\mathcal P(\mM))$ where we from now on denote \[U_o=\{y_{2,2},\ldots,y_{n-1,n-1}\}\] and $U_c=\{y_{i,j}|i>j\}$. As we know, the integer points in $\Pi_{U_o,U_c}(\mathcal P(\mM))$ are given by order ideals $J\subset\mathcal P(\mM)$.
\begin{proposition}
For any order ideal $J\subset\mathcal P(\mM)$ the set $K_{U_o,U_c}(J)$ consists of the elements $y_{i,i}$ with $i\le k$ for some $1\le k\le n-1$ and a sequence of elements $(y_{r_1,s_1},\ldots,y_{r_m,s_m})$ such that $r_m<\ldots<r_1\le k<s_1<\ldots<s_m$.
\end{proposition}
\begin{proof}
This is a direct consequence of Proposition~\ref{combdesc}. Here $J\cap U_o$ consists of the elements $y_{2,2},\ldots,y_{k,k}$ and $k=1$ if the intersection is empty. The elements $y_{r_i,s_i}$ are those elements of $U_c$ that are maximal in $J$, they form an antichain in $\mathcal P(\mM)$. Furthermore, either $y_{k,k}$ is incomparable to $y_{r_1,s_1}$, in which case $r_1<k<s_1$, or $y_{k,k}<y_{r_1,s_1}$ and then $r_1=k<s_1$ (since $y_{k+1,k+1}\not\le y_{r_1,s_1}$).
\end{proof}

For $a\in\mM$ let us define a sequence $\nu(a)=(\nu_1(a),\ldots,\nu_{\kappa(a)}(a))$ as follows. Let the integer $k$ and the sequence $(y_{r_1,s_1},\ldots,y_{r_m,s_m})$ be as in the above proposition applied to $J=\iota_\mM(a)$. Then $\kappa(a)=k$ and $\nu_i(a)=i$ unless $i\in\{r_1,\ldots,r_m\}$. If, however, $i=r_j$ then we set $\nu_i(a)=s_j$. In other words, we take the sequence $(1,\ldots,k)$ and then replace every $r_j$ with $s_j$. 
\begin{example}
Let $n=4$ as in the above diagram. When $\iota_\mM(a)=\{y_{1,2},y_{2,2},y_{1,3},\\y_{2,3},y_{1,4}\}$ we have $k=2$, $m=2$, $(y_{r_1,s_1},y_{r_2,s_2})=(y_{2,3},y_{1,4})$ and $\nu(a)=(4,3)$. When $\iota_\mM(a)=\{y_{1,2},y_{2,2},y_{1,3}\}$ we have $k=2$, $m=1$, $y_{r_1,s_1}=y_{1,3}$ and $\nu(a)=(3,2)$. When $\iota_\mM(a)=\varnothing$ we have $k=1$, $m=0$ and $\nu(a)=(1)$. (The elements $a\in\mM$ themselves are uniquely determined by $\iota_\mM(a)$ and can be computed but are not relevant to this example.)
\end{example}

Now we are ready to recall the definition of PBW-semistandard Young tableaux which were introduced in~\cite{Fe}. All tableaux in consideration are on the alphabet $[1,n]$. A Young tableau consisting of one column $(\alpha_1,\ldots,\alpha_k)$ is PBW-semistandard if $\alpha_r\le k$ implies $\alpha_r=r$ and for $r_1<r_2$ such that $\alpha_{r_1}>k$ and $\alpha_{r_2}>k$ we must have $\alpha_{r_1}>\alpha_{r_2}$. A Young tableau with two columns $(\alpha_1,\ldots,\alpha_k)$ and $(\beta_1,\ldots,\beta_l)$ (in that order) where $k\ge l$ is PBW-semistandard if each column is PBW-semistandard by itself and for every $\beta_r$ there exists a $s\ge r$ such that $\alpha_s\ge \beta_r$. An arbitrary Young tableau is semistandard if any of its two consecutive columns form a PBW-semistandard Young tableau. A simple observation that we will use repeatedly is that if the one-column tableau $(\alpha_1,\ldots,\alpha_k)$ is PBW-semistandard, then for every $j$ we have $\alpha_j\ge j$. Hence in the above definition of a two-column PBW-semistandard tableau it is sufficient to require the existence of $s\ge r$ with $\alpha_s\ge \beta_r$ only for $\beta_r>l$. 

\begin{lemma}\label{ordersame}
We have $a\ge b$ in $\mM$ if and only if the two-column Young tableau with columns $\nu(a)$ and $\nu(b)$ (in this order) is PBW-semistandard.
\end{lemma}
\begin{proof}
It is immediate from the definition of $\nu$ that the tableau with a single column $\nu(a)$ is PBW-semistandard for any $a\in\mM$. For any PBW-semistandard one-column tableau $(\alpha_1,\ldots,\alpha_k)$ one obtains an $a$ with $\nu(a)=(\alpha_1,\ldots,\alpha_k)$ by setting $a=\iota_\mM^{-1}(J)$ for the minimal order ideal $J$ containing $y_{k,k}$ and all $y_{j,\alpha_j}$ with $\alpha_j>j$.

Now, recall that $a\ge b$ if and only if $\iota_\mM(b)\subset\iota_\mM(a)$. The latter holds if and only if for any maximal $y$ in $\iota_\mM(b)$ there exists a maximal $y'$ in $\iota_\mM(a)$ such that $y'\ge y$. Consider $a,b\in\mM$ and denote \[K_{U_o,U_c}(\iota_\mM(a))=\{y_{i,i}|i\le\kappa(a)\}\sqcup\{y_{r_1,s_1},\ldots,y_{r_m,s_m}\},\]\[K_{U_o,U_c}(\iota_\mM(b))=\{y_{i,i}|i\le\kappa(b)\}\sqcup\{y_{u_1,v_1},\ldots,y_{u_l,v_l}\}.\]

The maximal elements in $K_{U_o,U_c}(\iota_\mM(a))$ are $y_{\kappa(a),\kappa(a)}$ and all the $y_{r_j,s_j}$ and similarly for $K_{U_o,U_c}(\iota_\mM(b))$. In particular, some element of $\iota_\mM(a)$ is greater than or equal to $y_{\kappa(b),\kappa(b)}$ if and only if $\kappa(a)\ge\kappa(b)$, since all $r_i\le\kappa(a)$. Furthermore, some element of $\iota_\mM(a)$ is greater than or equal to $y_{u_j,v_j}$ if and only if $v_j\le\kappa(a)$ or $r_l\ge u_j$ and $s_l\ge v_j$ for some $l$. This is equivalent to the existence of a $t\ge u_j$ such that $\nu(a)_t\ge \nu(b)_{u_j}=v_j$. We have seen that, indeed, $\nu(a)$ and $\nu(b)$ form a PBW-semistandard Young tableau if and only if every maximal element in $\iota_\mM(b)$ is less than or equal to some maximal element in $\iota_\mM(a)$.
\end{proof}

Let us consider the poset $(\mN,<)$ defined as follows. The elements of $\mN$ are all possible $b_{\alpha_1,\ldots,\alpha_k}$ such that $1\le k\le n-1$ and $(\alpha_1,\ldots,\alpha_k)$ is a one-column PBW-semistandard tableau. The order relation is given by $b_{\beta_1,\ldots,\beta_l}\le b_{\alpha_1,\ldots,\alpha_k}$ whenever the columns $(\alpha_1,\ldots,\alpha_k)$ and $(\beta_1,\ldots,\beta_l)$ (in this order!) form a PBW-semistandard tableau. We have established the following.

\begin{theorem}\label{M=L}
The map $\tau:a\mapsto b_{\nu(a)}$ is an isomorphism between the posets $(\mM,<)$ and $(\mN,<)$.
\end{theorem}
\begin{proof}
The order ideal $\iota_\mM(a)$ and $a$ itself are uniquely determined by the sequence $\nu(a)$, therefore, $\tau$ is injective. All the elements of a one-column PBW-semistandard tableau are distinct and for every $1\le i_1<\ldots<i_k\le n$ there is exactly one one-column PBW-semistandard tableau with elements $\{i_1,\ldots,i_k\}$, therefore, $|\mN|=|\mM|$. The rest follows from Lemma~\ref{ordersame}.
\end{proof}

\begin{example}\label{hasse}
Below are the Hasse diagrams of $\mM$ (left) and $\mN$ (right) for $n=4$ with $i_1\dots i_k$ in place of $a_{i_1,\dots,i_k}$ and $\alpha_1\dots\alpha_k$ in place of $b_{\alpha_1,\dots,\alpha_k}$. The join-irreducible elements in both lattices are shown in red. Here the isomorphism $\tau$ simply translates the left diagram onto the right one.
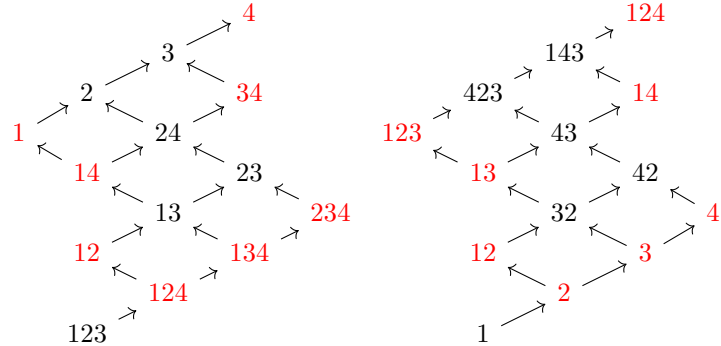
\begin{center}
\begin{tikzcd}[row sep=.5mm,column sep=2.5mm]
&&&\color{red}4\\
&&3\arrow[ur]\\
&2\arrow[ur]&&\color{red}34\arrow[ul]\\
\color{red}1\arrow[ur]&&24\arrow[ur]\arrow[ul]\\
&\color{red}14\arrow[ur]\arrow[ul]&&23\arrow[ul]\\
&&13\arrow[ur]\arrow[ul]&&\color{red}234\arrow[ul]\\
&\color{red}12\arrow[ur]&&\color{red}134\arrow[ur]\arrow[ul]\\
&&\color{red}124\arrow[ur]\arrow[ul]\\
&123\arrow[ur]
\end{tikzcd}
\begin{tikzcd}[row sep=.5mm,column sep=2.5mm]
&&&\color{red}124\\
&&143\arrow[ur]\\
&423\arrow[ur]&&\color{red}14\arrow[ul]\\
\color{red}123\arrow[ur]&&43\arrow[ur]\arrow[ul]\\
&\color{red}13\arrow[ur]\arrow[ul]&&42\arrow[ul]\\
&&32\arrow[ur]\arrow[ul]&&\color{red}4\arrow[ul]\\
&\color{red}12\arrow[ur]&&\color{red}3\arrow[ur]\arrow[ul]\\
&&\color{red}2\arrow[ur]\arrow[ul]\\
&1\arrow[ur]
\end{tikzcd}
\end{center}
\end{example}

\begin{remark}
One can notice a slight inconsistency in the above, namely, that semistandard tableaux correspond to weakly \emph{increasing} sequences in $(\mM,<)$ while PBW-semistandard tableaux correspond to weakly \emph{decreasing} sequences in $(\mN,<)$. Of course, one could say that $\mM$ is easily seen to be self-dual and the lattices are isomorphic regardless of the directions of the order relations. But there does not seem to be a natural way to redefine $\nu$ that would make $\tau$ an isomorphism between $(\mM,<)$ the dual of $(\mN,<)$. Indeed, that would require the odd-looking condition $\nu(\iota^{-1}(\varnothing))=(1,\ldots,n-2,n)$. Perhaps, the right way around this is to redefine PBW-semistandard tableaux.
\end{remark}

Note that $\mathcal P(\mN)$ is the image $\tau(\mathcal P(\mM))$. Define a partition of $\mathcal P(\mN)$ by setting $V_o=\tau(U_o)$ and $V_c=\tau(U_c)$. Our next goal is to show that the ideal $I^{V_o,V_c}(\mN)$ is an initial ideal of the ideal of Pl\"ucker relations. In fact, the subvariety it cuts out in the product $\bP$ is the toric variety associated with the FFLV polytope. We will not give a precise definition of this polytope, instead we will make use of a certain characterization of the toric ideal found in~\cite{fafefom}.

Let $I$ be the ideal of Pl\"ucker relations (as in Section~\ref{semistandard}) and let $\mM$ and $\mN$ be as in the previous section. First of all, we identify $R(\mN)$ with $R(\mM)$ by setting $X_{b_{\alpha_1,\ldots,\alpha_k}}=X_{\alpha_1,\ldots,\alpha_k}$. We will now write $R$ to denote either ring. 

Consider variables $z_{i,j}$ with $1\le i<j\le n$ and also variables $z_k$ with $1\le k\le n-1$. Consider the homomorphism $\psi:R\to\bC[\{z_{i,j}\},\{z_k\}]$ mapping $X_{\alpha_1,\ldots,\alpha_k}$ with $b_{\alpha_1,\ldots,\alpha_k}\in\mN$ to \[z_k\prod_{j|\alpha_j>k}z_{j,\alpha_j}.\] Let $I^{fflv}$ denote the kernel of $\psi$. The following fact is proved in~\cite{fafefom} (see Theorem 5.1 and its proof).
\begin{theorem}[\cite{fafefom}]
The ideal $I^{fflv}$ has the form $\initial_w I$ for some $w\in\bR^{\mN}$. The subvariety cut out by this ideal in the product $\bP$ is the toric variety associated with the Feigin--Fourier--Littelmann--Vinberg polytope of~\cite{FFL1}.
\end{theorem}

We interpret $I^{fflv}$ as a generalized Hibi ideal.    
\begin{theorem}
$I^{fflv}$ coincides with the ideal $I^{V_o,V_c}(\mN)$.
\end{theorem}
\begin{proof}
Recall the map $\theta_{V_o,V_c}:R\to\bC[\{z_p,p\in\mathcal P(\mN)\},t]$ with kernel $I^{V_o,V_c}(\mN)$ that was used to define the generalized Hibi ideal. One sees that maps $\psi$ and $\theta_{V_o,V_c}$ can be identified by setting $z_{\tau(y_{r,s})}=z_{r,s}$ for $r<s$, $z_{\tau(y_{k,k})}=z_k/z_{k-1}$ and $t=z_1$.
\end{proof}

\begin{remark}
It was originally anticipated by the author that the ideal $I^{fflv}$ will be an ideal of the type considered in~\cite{HL}, i.e.\ will be equal to $I^{\varnothing,P(\mN)}(\mN)$. That is since the analog of this statement holds in the simpler Grassmannian case. Of course, for the complete flag variety this turns out to be false, for instance, when $n=4$ one can see that $X_{1,2,3}X_{1,4}-X_{1,4,3}X_1$ is among the quadratic generators of $I^{\varnothing,P(\mN)}(\mN)$ meaning that the ideal is not $\deg$-homogeneous and thus not an initial ideal of $I$. It would be interesting to find other natural bijections between the set of Pl\"ucker variables and the lattice $\mN\simeq\mM$ such that one of the generalized Hibi ideals would be an initial ideal of $I$.
\end{remark}

\begin{remark}
As we have mentioned, the subvarieties cut out in the product $\bP$ by the ideals $I^h(\mM)$ and $I^{V_o,V_c}(\mN)$ are the toric varieties associated with, respectively, the Gelfand-Tsetlin and the FFLV polytopes. It also follows from Remark~\ref{generalizedtoric} that the varieties cut out by these two ideals in, respectively, $\bP(\bC^\mM)$ and $\bP(\bC^\mN)$ are the toric varieties associated with the corresponding interpolating polytopes. To avoid confusion let us point out that the former two toric varieties are different from the latter two and the Gelfand--Tsetlin and FFLV polytopes are different from the interpolating polytopes. In particular, in the former case the varieties and polytopes have dimension $\frac{n(n-1)}2$ and in the latter case the dimension is $|\mathcal P(\mM)|=|\mathcal P(\mN)|=\frac{n^2+n-4}2$. However, it is shown in~\cite{ABS} that the Gelfand-Tsetlin and FFLV polytopes can be obtained as the \emph{marked order polytope} and the \emph{marked chain polytope} of a poset very similar to $\mathcal P(\mM)$.
\end{remark}

\section{An algebra with straightening laws}\label{ASL}

In this section we will show how the notion of a PBW-semistandard tableau induces the structure of an algebra with straightening laws on the Pl\"ucker algebra $R/I$. Consider an arbitrary finite poset $(\Omega,\le)$ and a commutative $\bC$-algebra $\mathcal R$ with an injection $\varepsilon:\Omega\to\mathcal R$. We say that $\mathcal R$ is an \emph{algebra with straightening laws} (or ASL) on $\Omega$ if the following two properties hold.
\begin{itemize}
\item[ASL1.] The set of products $\varepsilon(\omega_1)\dots\varepsilon(\omega_m)$ where $\omega_1\le\dots\le\omega_m$ form a basis in $\mathcal R$. Such products are known as standard monomials.
\item[ASL2.] For any incomparable $\chi_1,\chi_2\in\Omega$ consider the unique representation of $\varepsilon(\chi_1)\varepsilon(\chi_2)$ as a linear combination of standard monomials. For any standard monomial $\varepsilon(\omega_1)\dots\varepsilon(\omega_m)$ with $\omega_1\le\dots\le\omega_m$ occurring in this linear combination we have $\omega_1\le\chi_i$ for both $i$.
\end{itemize}

In the above definition we follow~\cite{H}. The notion of an algebra with straightening laws is closely related to and sometimes used interchangeably with the notion of a~\emph{Hodge algebra}, originally both terms are due to De Concini, Eisenbud and Procesi. In their work~\cite{DEP2} the structure defined above is termed an \emph{ordinal} Hodge algebra, while their definition of a Hodge algebra is slightly more general.

One can immediately observe that the injection $\varepsilon:\mM\to R/I$ mapping $a_{i_1,\dots,i_k}$ to the class of $X_{i_1,\dots,i_k}$ turns $R/I$ into an ASL over $\mM$. Indeed, as we have discussed, property ASL1 is a classical fact while ASL2 is implied by Theorem~\ref{strlaws}(c). In fact, Theorem~\ref{strlaws}(c) is somewhat stronger, it shows that if we reverse the order relation on $\mM$ the map $\varepsilon$ still makes $R/I$ an ASL over $\mM$. Historically, this construction was one of the first examples of an ASL (at least when restricted to the Grassmannian case) and motivated the definition.   

Our current goal is to prove an analog of Theorem~\ref{strlaws}, in particular, to show that the injection $\varepsilon':\mN\to R/I$ mapping $b_{\alpha_1,\dots,\alpha_k}$ to the class of $X_{\alpha_1,\dots,\alpha_k}$ turns $R/I$ into an ASL over $\mN$. 

Consider the ideal $I^m(\mN)$, it is spanned by monomials \[X_{\alpha_1^1,\ldots,\alpha_{k_1}^1}\dots X_{\alpha_1^m,\ldots,\alpha_{k_m}^m}\] such that each $(\alpha_m^1,\ldots,\alpha_{k_m}^m)$ is a PBW-semistandard tableau with one column but these $m$ columns cannot be arranged into a PBW-semistandard tableau. It is shown in~\cite{Fe} (and also follows from the fact that $I^m(\mN)$ is an initial ideal of $I$) that the monomials not contained in $I^m(\mN)$ project to a basis in $R/I$. This establishes property ASL1 for $\varepsilon'$.

For every monomial $M\in I^m(\mN)$ we have a unique straightening relation $M+Q\in I$ where $Q$ is a linear combination of monomials not contained in $I^m(\mN)$. In particular, for incomparable $a$ and $b$ in $\mN$ this element can be written as \[e(a,b)=X_aX_b-\sum_{i=0}^{m(a,b)}c_i(a,b)X_{g_i(a,b)}X_{h_i(a,b)}\] where the pairs $(g_i(a,b),h_i(a,b))\in\mN^2$ are pairwise distinct, $g_i(a,b)< h_i(a,b)$ and $c_i(a,b)\neq 0$. To prove our analog of Theorem~\ref{strlaws} we will need the following lemma.
\begin{lemma}[see {\cite[Equation {$(*)'$}]{DEP}}]\label{altsumlemma}
Consider two Pl\"ucker variables $X_{i_1,\dots,i_k}$ and $X_{j_1,\dots,j_l}$ with $k\ge l$ and an integer $1\le r\le l$. Then the we have 
\begin{equation}\label{altsum}
\sum_\sigma (-1)^\sigma X_{i_1,\dots,i_{r-1},\sigma(i_r),\dots,\sigma(i_k)}X_{\sigma(j_1),\dots,\sigma(j_r),j_{r+1},\dots,j_l}\in I
\end{equation}
where the summation ranges over all permutations $\sigma$ of the set of symbols $\{j_1,\ldots,j_r,\\i_r,\dots,i_k\}$.
\end{lemma}

For the chosen $V_o,V_c$ let us use the shorthand notation $\odot=\odot_{V_o,V_c}$.
\begin{theorem}\label{pbwstrlaws}
\hfill
\begin{enumerate}[label=(\alph*)]
\item The quadratic straightening relations $e(a,b)$ form a minimal generating set of $I$.
\item WLOG we may assume that $g_0(a,b)=a\odot b$, $h_0(a,b)=a\lor b$ and $c_0(a,b)=1$.
\item For $1\le i\le m(a,b)$ we have $g_i(a,b)\le a\land b$ (not $a\odot b$!) and $h_i(a,b)> a\lor b$. In particular, $R/I$ is an ASL over $\mN$ with respect to the injection $\varepsilon'$.
\end{enumerate}
\end{theorem}
\begin{proof}
The relations $e(a,b)$ are linearly independent because for every incomparable pair $\{a,b\}$ the monomial $X_aX_b$ occurs in exactly one of these relations. Their number is equal to the number of incomparable pairs in $\mN$ (or $\mM$) which is the dimension of the quadratic part of $I$. Part (a) follows. 

Now, in view of Corollary~\ref{quadgens}, the expression $X_aX_b-X_{a\odot b}X_{a\lor b}$ lies in $I^{V_o,V_c}(\mN)$. Consequently, for any $w$ with $\initial_w I=I^{V_o,V_c}(\mN)$ the initial part $\initial_we(a,b)$ must equal $X_aX_b-X_{a\odot b}X_{a\lor b}$ which implies part (b).

Part (c) is proved with the use of Lemma~\ref{altsumlemma} which provides an algorithm for computing the expressions $e(a,b)$ (that is a \textit{straightening algorithm}). Let $a=b_{\alpha_1,\dots,\alpha_k}$ and $b=b_{\beta_1,\dots,\beta_l}\in\mN$ be incomparable with $k\ge l$. This means that we have an $r\le l$ such that $\alpha_j<\beta_r$ for all $j\ge r$. Consider the expression
\begin{equation}\label{fflvaltsum}
\sum_\sigma (-1)^\sigma X_{\alpha_1,\dots,\alpha_{r-1},\sigma(\alpha_r),\dots,\sigma(\alpha_k)}X_{\sigma(\beta_1),\dots,\sigma(\beta_r),\beta_{r+1},\dots,\beta_l}\in I
\end{equation}
as in Lemma~\ref{altsumlemma}. First, let $\sigma$ preserve the set $\{\beta_1,\dots,\beta_r\}$ (and also the set $\{\alpha_r,\dots,\alpha_k\}$). Then one sees that \[X_{\alpha_1,\dots,\alpha_{r-1},\sigma(\alpha_r),\dots,\sigma(\alpha_k)}X_{\sigma(\beta_1),\dots,\sigma(\beta_r),\beta_{r+1},\dots,\beta_l}=(-1)^\sigma X_{\alpha_1,\dots,\alpha_k}X_{\beta_1,\dots,\beta_l}\] and, therefore, $X_{\alpha_1,\dots,\alpha_k}X_{\beta_1,\dots,\beta_l}$ appears in~\eqref{fflvaltsum} with a positive coefficient (i.e.\ $r!(k-r+1)!$). Now consider one of the remaining $\sigma$ and let \[\{\alpha_1,\dots,\alpha_{r-1},\sigma(\alpha_r),\dots,\sigma(\alpha_k)\}=\{\gamma_1,\ldots,\gamma_k\},\]\[\{\sigma(\beta_1),\dots,\sigma(\beta_r),\beta_{r+1},\dots,\beta_l\}=\{\delta_1,\dots,\delta_l\}\] for one-column PBW-semistandard tableaux $(\gamma_1,\ldots,\gamma_k)$ and $(\delta_1,\dots,\delta_l)$. We claim that $b_{\gamma_1,\ldots,\gamma_k}> b_{\alpha_1,\ldots,\alpha_k}$ and $b_{\delta_1,\dots,\delta_l}< b_{\beta_1,\ldots,\beta_l}$. The $k+1$ elements that are being permuted are divided into 4 groups:
\begin{enumerate}[label=(\roman*)]
\item the smallest elements are the $\beta_j$ with $\beta_j=j$, they are smaller than $r$,
\item then follow the $\alpha_j=j$, they lie in $[r,k]$, 
\item then the $\alpha_j> k$, they lie in $[k+1,\beta_r-1]$,
\item the largest elements are the $\beta_j>l$, they are no less than $\beta_r$.  
\end{enumerate}
In particular, these $k+1$ elements are pairwise distinct. The set $\{\delta_1,\dots,\delta_l\}$ is obtained from $\{\beta_1,\dots,\beta_l\}$ by replacing some of the elements in groups (i) and (iv) with an equal number of elements in groups (ii) and (iii). The set of $j\le r$ such that $\delta_j>l$ contains the set of $j'\le r$ such that $\beta_{j'}>l$ (these $\beta_{j'}$ make up group (iv)). Several largest among the $\delta_j>l$ with $j\le r$ (i.e.\ those with the smallest $j$) are those elements of group (iv) that were not exchanged. These last two facts show that if for $j\le r$ and $j'\le r$ we have $\delta_j=\beta_{j'}>l$, then $j\le j'$. All other $\delta_j>l$ with $j\le r$ either came from groups (ii) and (iii) or are equal to some $\beta_j$ with $j>r$. All such elements are no greater than $\beta_r$. We have verified the defining condition of PBW-semistandard tableaux for columns $(\beta_1,\dots,\beta_l)$ and $(\delta_1,\dots,\delta_l)$ for $\delta_j$ with $j\le r$. 

Now, the set of $j>r$ with $\delta_j>l$ is contained in the set of $j'$ with $\beta_{j'}>l$. The set of $\delta_j>l$ with $j>r$ consists of several smallest $\beta_j>l$ and several smallest of the $\alpha_j>l$ that were permuted. Let $j_i$ be the $i$th largest $j$ with $\delta_j>j>r$ and $j'_i$ be the $i$th largest $j'$ with $\beta_{j'}>j'>r$. We see that $j'_i\ge j_i$ and $\beta_{j'_i}\ge\delta_{j_i}$ for all $j_i$. We have shown that $b_{\delta_1,\dots,\delta_l}< b_{\beta_1,\ldots,\beta_l}$.

The proof that $b_{\gamma_1,\ldots,\gamma_k}> b_{\alpha_1,\ldots,\alpha_k}$ is similar. One sees that the set of $j<r$ with $\gamma_j>k$ is contained in the set of $j'<r$ with $\alpha_{j'}>k$. The set of $\gamma_j>k$ with $j<r$ consists of several largest $\alpha_j>k$ with $j<r$ while the remaining $\alpha_j>k$ with $j<r$ are equal to some $\gamma_j$ with $j\ge r$. We see that for every $\alpha_j>k$ with $j<r$ we have $\gamma_{j'}=\alpha_j$ for some $j'\ge j$. The set of $j\ge r$ with $\gamma_j>j$ contains the set of $j'\ge r$ with $\alpha_{j'}>j$. Several smallest $\gamma_j>k$ with $j\ge r$ are elements of group (iii). The remaining (larger) $\gamma_j>k$ with $j\ge r$ are made up of elements of group (iv) that were exchanged and several largest of the $\alpha_j>k$ with $j<r$. Let $j_i$ be the $i$th largest $j$ with $\gamma_j>j\ge r$ and $j'_i$ be the $i$th largest $j'$ with $\alpha_{j'}>j'>r$ (i.e.\ with $\alpha_{j'}$ in group (iii)). We see that $j_i\ge j'_i$ and $\gamma_{j_i}\ge\alpha_{j'_i}$ for all $j'_i$.

Next, if for some $\sigma$ the monomial 
\begin{equation}\label{monomial}
X_{\alpha_1,\dots,\alpha_{r-1},\sigma(\alpha_r),\dots,\sigma(\alpha_k)}X_{\sigma(\beta_1),\dots,\sigma(\beta_r),\beta_{r+1},\dots,\beta_l}
\end{equation}
is non-standard (lies in $I^m(\mN)$) we may again apply Lemma~\ref{altsumlemma} to this pair of tuples and subtract the result from~\eqref{fflvaltsum} with the coefficient needed to cancel out the monomial \eqref{monomial}. We may repeat this operation until there are no non-standard monomials left in the expression other than $X_{\alpha_1,\dots,\alpha_k}X_{\beta_1,\dots,\beta_l}$. This algorithm will terminate since every step replaces a monomial $X_cX_d$ with a sum of monomials $X_{c'}X_{d'}$ with $c'> c$ and $d'< d$ (we assume that we never transpose the first and second variable in any of the monomials during the procedure).

As a result we obtain the expression $e(a,b)$. We see that every $g_i(a,b)< b$ and every $h_i(a,b)> a$. Now first suppose that $k=l$. Then we may transpose $a$ and $b$ and repeat the straightening algorithm for this reversed pair. We must end up with the same expression, in particular, we see that every $g_i(a,b)< a$ and every $h_i(a,b)> b$. This proves part (c) when $k=l$ (we have strictly $h_i(a,b)>a\lor b$ for $i\ge 1$ since $h_i(a,b)=a\lor b=h_0(a,b)$ would imply $g_i(a,b)=g_0(a,b)$ and $i=0$).

Now suppose that $k>l$. We consider two cases. First suppose that none of the $\alpha_j$ or $\beta_j$ lie in $[l+1,k]$. We apply Lemma~\ref{k>l} to the relation $e(a,b)$ by adding the subscripts $l+1,\dots,k$ to the $l$-tuples. It is evident that $b_{\alpha_1,\dots,\alpha_k}$ is incomparable to $b_{\beta_1,\dots,\beta_l,l+1,\dots,k}$. It is also clear that for any standard monomial $X_{\alpha'_1,\dots,\alpha'_k}X_{\beta'_1,\dots,\beta'_l}$ appearing in $e(a,b)$ we have $b_{\alpha'_1,\dots,\alpha'_k}>b_{\beta'_1,\dots,\beta'_l,l+1,\dots,k}$. We see that the expression provided by Lemma~\ref{k>l} is precisely $e(a,b_{\beta_1,\dots,\beta_l,l+1,\dots,k})$. Having considered the $k=l$ case, we know that for any standard monomial $X_{\alpha'_1,\dots,\alpha'_k}X_{\beta'_1,\dots,\beta'_l,l+1,\dots,k}$ appearing in $e(a,b_{\beta_1,\dots,\beta_l,l+1,\dots,k})$ we have $b_{\alpha'_1,\dots,\alpha'_k}>b_{\beta_1,\dots,\beta_l,l+1,\dots,k}$ and $b_{\alpha_1,\dots,\alpha_k}>b_{\beta'_1,\dots,\beta'_l,l+1,\dots,k}$. These relations imply, respectively, $b_{\alpha'_1,\dots,\alpha'_k}>b_{\beta_1,\dots,\beta_l}$ and $b_{\alpha_1,\dots,\alpha_k}>b_{\beta'_1,\dots,\beta'_l}$ which proves part (c) for this case.

Suppose we are not in the above case. Since the expression $e(a,b)$ is independent of $n$ for large enough $n$ we can assume that all $\alpha_j$ and all $\beta_j$ are no greater than $n-k+l$. Consider the permutation $\rho\in S_n$ such that $\rho(j)=j$ for $j\le l$, $\rho(j)=j+k-l$ for $j\in[l+1,n-k+l]$ and $\rho(j)=j-n+k$ for $j>n-k+l$. This permutation acts on $R$ via $\rho:X_{i_1,\dots,i_m}\mapsto X_{\rho(i_1),\dots,\rho(i_m)}$. The automorphism $\rho$ preserves the ideal $I$, since it corresponds to the action of a Weyl group element. We also view $\rho$ as map from $\mN$ to itself via $\rho(b_{\gamma_1,\dots,\gamma_m})=b_{\gamma'_1,\dots,\gamma'_m}$ where $\{\gamma'_1,\dots,\gamma'_m\}=\{\rho(\gamma_1),\dots,\rho(\gamma_m)\}$ (so that $\rho(X_c)=\pm X_{\rho(c)}$). We show that $\rho(e(a,b))=\pm e(\rho(a),\rho(b))$, in particular, $\rho(a)$ and $\rho(b)$ are incomparable. Indeed, this follows directly from the following claim.

\textbf{Claim.} Let $b_{\gamma_1,\dots,\gamma_k},b_{\delta_1,\dots,\delta_l}\in\mN$ be such that all $\gamma_j$ and $\delta_j$ are no greater than $n-k+l$. Then $b_{\gamma_1,\dots,\gamma_k}>b_{\delta_1,\dots,\delta_l}$ if and only if $\rho(b_{\gamma_1,\dots,\gamma_k})>\rho(b_{\delta_1,\dots,\delta_l})$.

To prove the claim note that if $\rho(b_{\delta_1,\dots,\delta_l})=b_{\delta'_1,\dots,\delta'_l}$, then for any $j$ either $\delta'_j=\delta_j=j$ or $\delta'_j=\delta_j+k-l>k$. Similarly, if $\rho(b_{\gamma_1,\dots,\gamma_k})=b_{\gamma'_1,\dots,\gamma'_k}$, then for any $j\le l$ either $\gamma'_j=\gamma_j=j$ or $\gamma'_j=\gamma_j+k-l>k$. Meanwhile, the tuple $(\gamma'_{l+1},\dots,\gamma'_k)$ is a permutation of the tuple $(\gamma_{l+1}+k-l,\dots,\gamma_k+k-l)$ with all elements greater than $k$. For any $\delta_j>l$ it is now evident that there exists a $\gamma_{j'}\ge\delta_j$ with $j'\ge j$ if and only if there exists a $\gamma'_{j''}\ge\delta'_j=\delta_j+k-l$ with $j''\ge j$.

We see that $m(\rho(a),\rho(b))=m(a,b)$ and $g_i(\rho(a),\rho(b))=\rho(g_i(a,b))$ and $h_i(\rho(a),\rho(b))=\rho(h_i(a,b))$. Moreover, due to our choice of $\rho$ the pair $\{\rho(a),\rho(b)\}$ is covered by the previous case. This provides $g_i(\rho(a),\rho(b))<\rho(a)$ and $h_i(\rho(a),\rho(b))>\rho(b)$. By our proved claim this implies $g_i(a,b)<a$ and $h_i(a,b)>b$.
\end{proof}

\section{The PBW-semistandard maximal cone}

In this section we finally describe the cone $C(I,I^m(\mN))$. As a first step we have analogs of Propositions~\ref{diamondfacets} and~\ref{ssytredundant}. 
\begin{proposition}\label{pbwdiamondfacets}
For every diamond pair $\{a,b\}\subset\mN$ the hyperplane $w_a+w_b=w_{a\lor b}+w_{a\odot b}$ contains a facet of $C(I,I^m(\mN))$.
\end{proposition}
\begin{proof}
The proposition is deduced from Theorem~\ref{genhibimonfacets} the same way as Proposition~\ref{diamondfacets} from Theorem~\ref{hibimonfacets}.
\end{proof}

\begin{proposition}\label{pbwredundant}
The cone $C(I,I^m(\mN))$ is composed of all $w$ satisfying 
\begin{equation}\label{pbwstrineq}
w_a+w_b<w_{g_i(a,b)}+w_{h_i(a,b)}
\end{equation}
for all incomparable pairs $\{a,b\}\subset\mN$ and all $0\le i\le m(a,b)$.
\end{proposition}
\begin{proof}
This is proved the same way as Proposition~\ref{ssytredundant}.
\end{proof}

Next we need to establish several basic properties of the diamond pairs in $\mN$, similarly to the study in Section~\ref{semistandard}. Of course, in view of Theorem~\ref{M=L}, $\{a,b\}\subset\mN$ is a diamond pair if and only if $\{\tau^{-1}(a),\tau^{-1}(b)\}\subset\mM$ is a diamond pair. However, since the combinatorics of the map $\tau$ and of PBW-semistandard tableaux is quite complicated in general, we will often use a more graphic approach in terms of the poset $\mathcal P(\mN)$. The visualization of $\mathcal P(\mN)$ similar to the one given for $\mathcal P(\mM)$ in~\eqref{n=4}, see Example~\ref{n=7}.

As we have pointed out, $a\in\mN$ covers $b\in\mN$ if and only if the order ideal $\iota_\mN(a)$ is obtained from $\iota_\mN(b)$ by adding one element. Consequently, any diamond pair in $\{a,b\}\subset\mN$ is obtained by choosing a $c\in\mN$, choosing two distinct minimal elements $x$ and $y$ in the difference $\mathcal P(\mN)\backslash\iota_\mN(c)$ and defining $a$ and $b$ by $\iota_\mN(a)=\iota_\mN(c)\cup\{x\}$ and $\iota_\mN(b)=\iota_\mN(c)\cup\{y\}$. We then have $a\land b=c$ and $\iota_\mN(a\lor b)=\iota_\mN(c)\cup\{x,y\}$. 

Let us introduce the notation $x_{i,j}=\tau(y_{i,j})$ for the elements of $\mathcal P(\mN)$. We point out that Theorem~\ref{M=L} together with the definition of $\nu$ directly implies the following way of recovering $b\in\mN$ from the order ideal $\iota_\mN(b)\subset\mathcal P(\mN)$.
\begin{proposition}\label{tableaufromideal}
Let $b=b_{\alpha_1,\dots,\alpha_k}$. Then $k$ is the maximal $l$ such that $x_{l,l}\in\iota_\mN(b)$ or $k=1$ if no such $l$ exists. Furthermore, for every maximal element $x_{s,t}$ in $\iota_\mN(b)$ we have $\alpha_s=t$ and all remaining $\alpha_r$ are given by $\alpha_r=r$.
\end{proposition}

\begin{example}\label{n=7}
Here is a visualization of the poset $\mathcal P(\mN)$ for $n=7$.
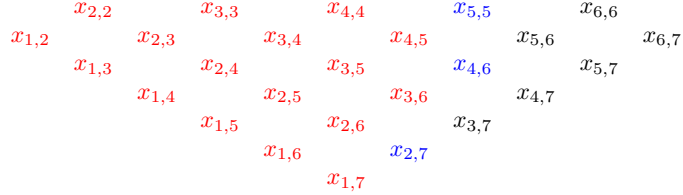
\begin{center}
\setcounter{MaxMatrixCols}{20}
\scalebox{0.9}{$\begin{matrix}
&\color{red}x_{2,2}&&\color{red}x_{3,3}&&\color{red}x_{4,4}&&\color{blue}x_{5,5}&&x_{6,6}&\\
\color{red}x_{1,2}&&\color{red}x_{2,3}&&\color{red}x_{3,4}&&\color{red}x_{4,5}&&x_{5,6}&&x_{6,7}\\
&\color{red}x_{1,3}&&\color{red}x_{2,4}&&\color{red}x_{3,5}&&\color{blue}x_{4,6}&&x_{5,7}&&\\
&&\color{red}x_{1,4}&&\color{red}x_{2,5}&&\color{red}x_{3,6}&&x_{4,7}&&&&\\
&&&\color{red}x_{1,5}&&\color{red}x_{2,6}&&x_{3,7}&&&&&&\\
&&&&\color{red}x_{1,6}&&\color{blue}x_{2,7}&&&&&&&&\\
&&&&&\color{red}x_{1,7}&&&&&&&&&&\\
\end{matrix}$}
\end{center}
\vspace{2mm}
The red elements compose an order ideal $\iota_\mN(c)$ and Proposition~\ref{tableaufromideal} shows that $c=b_{7,2,6,5}$. The blue elements are the minimal elements of $\mathcal P(\mN)\backslash\iota_\mN(c)$, therefore, $c$ is covered by precisely three elements corresponding to the order ideals $\iota(c)\cup\{x_{5,5}\}$, $\iota(c)\cup\{x_{4,6}\}$ and $\iota(c)\cup\{x_{2,7}\}$. Consequently, there exist three diamond pairs $\{a,b\}\subset\mN$ such that $a\land b=c$ (each composed of two elements that cover $c$).
\end{example}

We will say that a diamond pair $\{a,b\}\subset\mN$ is special if $\{\tau^{-1}(a),\tau^{-1}(b)\}$ is special. It turns out that special pairs have a simple characterization in terms of $\mathcal P(\mN)$ and its visualization.
\begin{proposition}\label{diagonalsquares}
For a diamond pair $\{a,b\}\subset\mN$ let $\iota_\mN(a)\backslash\iota_\mN(a\land b)=x_{s,t}$ and $\iota_\mN(b)\backslash\iota_\mN(a\land b)=x_{u,v}$ and assume (without loss of generality) that $u<s\le t<v$. Then $\{a,b\}$ is special if and only if $u=s-1$ and $v=t+1$.
\end{proposition}
\begin{proof}
Let $c=\tau^{-1}(a)$ and $d=\tau^{-1}(b)$. Note that $\iota_\mM(c)\backslash\iota_\mM(c\land d)=y_{s,t}$ and $\iota_\mM(d)\backslash\iota_\mM(c\land d)=y_{u,v}$. It suffices to show that $\{c,d\}$ is special if and only if $u=s-1$ and $v=t+1$.

Let $c\land d=a_{i_1,\dots,i_k}$. Now, if the diamond pair $\{c,d\}$ is of the type considered in possibility (1) in Proposition~\ref{diamondpairs}, then $c=a_{i'_1,\dots,i'_k}$ and $d=a_{i''_1,\dots,i''_k}$ where $i'_r=i_r$ for all $r$ except some value $r_2$ such that $i'_{r_2}=i_{r_2}+1$ and $i''_r=i_r$ for all $r$ except some value $r_1<r_2$ such that $i''_{r_1}=i_{r_1}+1$. By definition, $y_{s,t}$ is the only $y\in\mathcal P(\mM)$ such that $y\le c$ and $y\not\le c\land d$. From Proposition~\ref{irreducible} one then sees that $y_{s,t}=a_{j_1,\dots,j_{n-i_{r_2}+r_2-1}}$ where $j_r=r$ for $r<r_2$ and $j_r=r+i_{r_2}-r_2+1$ for $r\ge r_2$, in particular, $j_{r_2}=i_{r_2}+1=i'_{r_2}$ and $j_{n-i_{r_2}+r_2-1}=n$. We deduce that $s=i_{r_2}-r_2+1$ and $t=n-r_2+1$ (note that $s<t$). Similarly, we have $u=i_{r_1}-r_1+1$ and $v=n-r_1+1$. By Proposition~\ref{specialpairs} the pair $\{c,d\}$ is special if and only if $r_1=r_2-1$ (that is $v=t+1$) and $i_{r_2}=i_{r_1}+2$ (that is $u=s-1$).

If $\{c,d\}$ is of the type considered in possibility (2) in Proposition~\ref{diamondpairs}, then $i_k=n$ and $d=a_{i_1,\dots,i_{k-1}}$ while $c=a_{i'_1,\dots,i'_k}$ where $i'_r=i_r$ for all $r$ except some value $r_1$ such that $i'_{r_1}=i_{r_1}+1$. We have $y_{s,t}=a_{1,\dots,k-1}$ so that $s=t=n-k+1$. Also, as above, we deduce that $u=i_{r_1}-r_1+1$ and $v=n-r_1+1$. In this case $\{c,d\}$ is special if and only if $r_1=k-1$ and $i_{r_1}=n-2$ which means that $u=n-k$ and $v=n-k+2$.
\end{proof}
In terms of our visualization we see that a diamond pair $\{a,b\}\subset\mN$ is special if and only if one of the elements $\iota_\mN(a)\backslash\iota_\mN(a\land b)$ and $\iota_\mN(b)\backslash\iota_\mN(a\land b)$ is situated immediately above the other: it lies on the same vertical line but two rows higher. For instance, of the three diamond pairs obtained in Example~\ref{n=7} the only special one consists of $\iota_\mN^{-1}(\iota_\mN(c)\cup\{x_{5,5}\})$ and $\iota_\mN^{-1}(\iota_\mN(c)\cup\{x_{4,6}\})$. 
\begin{proposition}\label{p1q1}
Let $\{c,d\}\subset\mM$ be a special diamond pair and let $\iota_\mM(c)\backslash\iota_\mM(c\land d)=\{y_{s,t}\}$ and $\iota_\mM(d)\backslash\iota_\mM(c\land d)=\{y_{s-1,t+1}\}$. Then $\iota_\mM(c\land d)\backslash\iota_\mM(p_1(c,d))=\{y_{s-1,t}\}$ and $\iota_\mM(q_1(c,d))\backslash\iota_\mM(c\lor d)=\{y_{s,t+1}\}$.
\end{proposition}
\begin{proof}
Below is the corresponding fragment of $\mathcal P(\mM)$ where $J=\iota_\mM(c\land d)$ consists of the elements to the left of the dashed line. Note that $\iota_\mM(c)$ and $\iota_\mM(d)$ are obtained from $J$ by adding $y_{s,t}$ and $y_{s-1,t+1}$ respectively. We are to show that $\iota_\mM(p_1(c,d))$ is obtained from $J$ by removing $y_{s-1,t}$ and that $\iota_\mM(q_1(c,d))$ is obtained from $J$ by adding $y_{s,t}$, $y_{s-1,t+1}$ and $y_{s,t+1}$. 

\vspace{2mm}
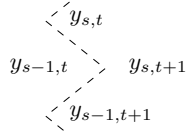
\begin{center}
\scalebox{0.9}{
\begin{picture}(100,100)
\put(25,100){$y_{s,t}$}
\put(0,80){$y_{s-1,t}$}
\put(50,80){$y_{s,t+1}$}
\put(25,60){$y_{s-1,t+1}$}
\multiput(17.5,102)(5,4){2}{\line(5,4){2.5}}
\multiput(15,100)(5,-4){5}{\line(5,-4){2.5}}
\multiput(40,80)(-5,-4){5}{\line(-5,-4){2.5}}
\multiput(15,60)(5,-4){2}{\line(5,-4){2.5}}
\end{picture}
}
\end{center}
\vspace{-1.9cm}

Since $\iota_\mM(c\land d)$ does not contain $y_{s,t}$ and $y_{s-1,t+1}$, the element $y_{s-1,t}$ is maximal in $\iota_\mM(c\land d)$. Now suppose that $\iota_\mM(c\land d)\backslash\iota_\mM(p_1(c,d))$ consists of some element $y\neq y_{s-1,t}$ with $y$ maximal in $\iota_\mM(c\land d)$. In this case at least one of $y_{s,t}$ and $y_{s-1,t+1}$ is minimal in $\mathcal P(\mM)\backslash\iota_\mM(p_1(c,d))$. Indeed, since $y$ is incomparable to $y_{s-1,t}$, there are two possibilities. If the horizontal row containing $y$ lies above that containing $y_{s-1,t}$, then $y_{s-1,t+1}$ is minimal in $\mathcal P(\mM)\backslash\iota_\mM(p_1(c,d))$, if it lies below, then $y_{s,t}$ is minimal. Let $y'\in\{y_{s,t},y_{s-1,t+1}\}$ be minimal in the complement. Then $\iota_\mM(p_1(c,d))\cup\{y'\}$ is an order ideal that is strictly smaller than $\iota_\mM(c\lor d)$ and, therefore, $\iota_\mM(q_1(c,d))$. In other words, \[p_1(c,d)<\iota_\mM^{-1}(\iota_\mM(p_1(c,d))\cup\{y'\})<q_1(c,d)\] which contradicts Corollary~\ref{only4}. The description of $\iota_\mM(q_1(c,d))$ is obtained similarly by supposing that $\iota_\mM(q_1(c,d))\backslash\iota_\mM(c\lor d)$ consists of some element $y\neq y_{s,t+1}$.
\end{proof}
Let $\{a,b\}=\{\tau(c),\tau(d)\}$ be the single special diamond pair appearing in Example~\ref{n=7}. We see that $\iota_\mM(p_1(c,d))=\iota_\mM(c\land d)\backslash\{x_{4,5}\}$ and $\iota_\mM(q_1(c,d))=\iota_\mM(c\lor d)\cup\{x_{5,6}\}$. For special diamond pairs we will need an explicit description of the straightening relation.

\begin{proposition}\label{pbwspecialpairs}
Let $\{a,b\}\subset\mN$ be a special diamond pair with $a=\tau(c)$ and $b=\tau(d)$. Then $m(a,b)=1$ and, furthermore, $a\odot b=\tau(p_1(c,d))$, $g_1(a,b)=a\land b$ and $h_1(a,b)=\tau(q_1(c,d))$. 
\end{proposition}
\begin{proof}
Denote $a\land b=b_{\alpha_1,\dots,\alpha_k}$ and let $\iota_\mN(a)\backslash\iota_\mN(a\land b)=\{x_{s,t}\}$ and $\iota_\mN(b)\backslash\iota_\mN(a\land b)=\{x_{s-1,t+1}\}$. First suppose that $s>2$, the case $s=2$ will be considered at the end of the proof. The fact that both $x_{s,t}$ and $x_{s-1,t+1}$ are minimal in $\mathcal P(\mN)\backslash\iota_\mN(a\land b)$ means that $x_{s-1,t}$ is maximal in $\iota_\mN(a\land b)$ and there exists an element $x_{s-2,t'}$ with $t'>t$ which is also maximal in $\iota_\mN(a\land b)$. The latter element is obtained by considering $x_{s-2,t}\in\iota_\mN(a\land b)$ and then moving down and to the right until we reach a maximal element. 

\pagebreak
Below is the corresponding fragment of $\mathcal P(\mN)$ for the two cases we will be considering: when $t'>t+1$ (left) and $t'=t+1$ (right). The order ideal $J=\iota_\mN(a\land b)$ consists of the elements to the left of the dashed line. By Proposition~\ref{p1q1}, we obtain $\iota_\mN(\tau(p_1(c,d)))$ by removing $x_{s-1,t}$ from $J$. We also obtain $\iota_\mN(a)$ by adding $x_{s,t}$ to $J$ and $\iota_\mN(b)$ by adding $x_{s-1,t+1}$. By adding both $x_{s,t}$ and $x_{s-1,t+1}$ to $J$ we obtain $\iota_\mN(a\lor b)$ and when we also add $x_{s,t+1}$ we obtain $\iota_\mN(\tau(q_1(c,d)))$. We need to distinguish between the two cases because on the left $x_{s-1,t+1}$ is maximal in $\iota_\mN(b)$, $\iota_\mN(a\lor b)$ and $\iota_\mN(\tau(q_1(c,d)))$ and on the right it is not.

\vspace{5mm}
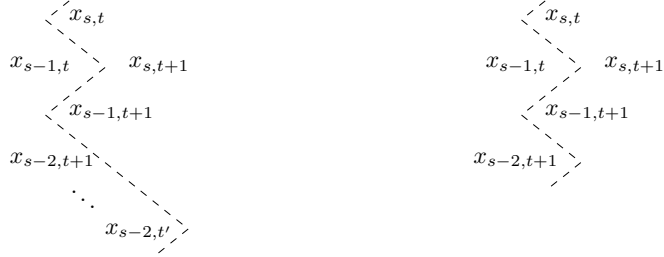
\begin{minipage}{.5\textwidth}
\begin{center}
\scalebox{0.9}{
\begin{picture}(100,100)
\put(25,100){$x_{s,t}$}
\put(0,80){$x_{s-1,t}$}
\put(50,80){$x_{s,t+1}$}
\put(25,60){$x_{s-1,t+1}$}
\put(0,40){$x_{s-2,t+1}$}
\put(25,20){$\ddots$}
\put(40,10){$x_{s-2,t'}$}
\multiput(17.5,102)(5,4){2}{\line(5,4){2.5}}
\multiput(15,100)(5,-4){5}{\line(5,-4){2.5}}
\multiput(40,80)(-5,-4){5}{\line(-5,-4){2.5}}
\multiput(15,60)(5,-4){12}{\line(5,-4){2.5}}
\multiput(75,12)(-5,-4){3}{\line(-5,-4){2.5}}
\end{picture}
}
\end{center}
\end{minipage}%
\begin{minipage}{.5\textwidth}
\begin{center}
\scalebox{0.9}{
\begin{picture}(100,100)
\put(25,100){$x_{s,t}$}
\put(0,80){$x_{s-1,t}$}
\put(50,80){$x_{s,t+1}$}
\put(25,60){$x_{s-1,t+1}$}
\put(-5,40){$x_{s-2,t+1}$}
\multiput(17.5,102)(5,4){2}{\line(5,4){2.5}}
\multiput(15,100)(5,-4){5}{\line(5,-4){2.5}}
\multiput(40,80)(-5,-4){5}{\line(-5,-4){2.5}}
\multiput(15,60)(5,-4){5}{\line(5,-4){2.5}}
\multiput(40,40)(-5,-4){3}{\line(-5,-4){2.5}}
\end{picture}
}
\end{center}
\end{minipage}
\vspace{1mm}

Now suppose that $\{c,d\}$ is of the type given by possibility (1) in Proposition~\ref{specialpairs}. As we have seen in the proof of Proposition~\ref{diagonalsquares}, this means that $s<t$, i.e.\ $x_{s,t}$ is not in the top row. In what follows we describe the elements $a\land b$, $\tau(p_1(c,d))$, $a$, $b$, $a\lor b$ and $\tau(q_1(c,d))$ by looking at the maximal elements in the corresponding order ideal and applying Proposition~\ref{tableaufromideal}. We see that $\alpha_{s-2}=t'$ and $\alpha_{s-1}=t$. Let $\tau(p_1(c,d))=b_{\alpha'_1,\dots,\alpha'_k}$, we have $\alpha'_r=\alpha_r$ unless $r=s-1$ and $\alpha'_{s-1}=s-1$. Denote $a=b_{\beta_1,\dots,\beta_k}$, then $\beta_r=\alpha_r$ unless $r=s$ or $r=s-1$ while $\beta_s=t$ and $\beta_{s-1}=s-1$. Denote $b=b_{\gamma_1,\dots,\gamma_k}$, then $\gamma_r=\alpha_r$ unless $r=s-1$ or $r=s-2$ and $\gamma_{s-1}=t+1$ while $\gamma_{s-2}$ depends on $t'$. If $t'=t+1$, then $\gamma_{s-2}=s-2$, if $t'>t+1$, then $\gamma_{s-2}=\alpha_{s-2}=t'$. Denote $a\lor b=b_{\delta_1,\dots,\delta_k}$, then $\delta_r=\gamma_r$ unless $r=s$  and $\delta_s=t$. Finally, denote $\tau(q_1(c,d))=b_{\delta'_1,\dots,\delta'_k}$, then $\delta'_r=\delta_r$ unless $r=s$ or $r=s-1$ while $\delta'_s=t+1$ and $\delta'_{s-1}=s-1$. We see that the expression 
\begin{equation}\label{pbwclassical}
X_{\beta_1,\dots,\beta_k}X_{\gamma_1,\dots,\gamma_k}-X_{\alpha'_1,\dots,\alpha'_k}X_{\delta_1,\dots,\delta_k}-X_{\alpha_1,\dots,\alpha_k}X_{\delta'_1,\dots,\delta'_k}
\end{equation}
is a classical Pl\"ucker relation given by Lemma~\ref{classical}. Indeed, since $\alpha_r=\alpha'_r=\beta_r=\gamma_r=\delta_r=\delta'_r$ unless $r\in\{s-2,s-1,s\}$, we may disregard all the subscripts in positions outside of $[s-1,s]$. Now, if $t'=t+1$, then~\eqref{pbwclassical} takes the form \[X_{t+1,s-1,t}X_{s-2,t+1,\alpha_s}-X_{t+1,s-1,\alpha_s}X_{s-2,t+1,t}-X_{t+1,t,\alpha_s}X_{s-2,s-1,t+1}\] which is seen to be a classical Pl\"ucker relation. If $t'>t+1$, then~\eqref{pbwclassical} takes the form \[X_{t',s-1,t}X_{t',t+1,\alpha_s}-X_{t',s-1,\alpha_s}X_{t',t+1,t}-X_{t',t,\alpha_s}X_{t',s-1,t+1}\] which is again a classical Pl\"ucker relation.

Suppose $\{c,d\}$ is of the type given by possibility (2) in Proposition~\ref{specialpairs}. This means that $s=t=k+1$. We have $\alpha_k=k+1$ and $\alpha_{k-1}=t'$. Similarly to the above, we denote $\tau(p_1(c,d))=b_{\alpha'_1,\dots,\alpha'_k}$, $a=b_{\beta_1,\dots,\beta_{k+1}}$, $b=b_{\gamma_1,\dots,\gamma_k}$, $a\lor b=b_{\delta_1,\dots,\delta_{k+1}}$ and $\tau(q_1(c,d))=b_{\delta'_1,\dots,\delta'_{k+1}}$ (whether the number of subscripts is $k$ or $k+1$ is determined using Proposition~\ref{tableaufromideal}). We have $\alpha_r=\alpha'_r=\beta_r=\gamma_r=\delta_r=\delta'_r$ for $r<k-1$. The remaining elements are defined as follows. We have $\alpha'_k=k$ and $\alpha'_{k-1}=t'$, we have $\beta_{k+1}=k+1$, $\beta_k=k$ and $\beta_{k-1}=t'$, we have $\gamma_k=k+2$ while $\gamma_{k-1}$ is equal to $k-1$ if $t'=t+1=k+2$ and to $t'$ if $t'>k+2$. We have $\delta_{k+1}=k+1$, $\delta_k=k+2$ and $\delta_{k-1}=\gamma_{k-1}$. Finally, $\delta'_{k+1}=k+2$, $\delta'_k=k$ and $\delta'_{k-1}=\gamma_{k-1}$. Again, we may disregard the subscripts in positions outside of $[k-1,k+1]$. If $t'=k+2$ we obtain the straightening relation \[X_{k+2,k,k+1}X_{k-1,k+2}-X_{k+2,k}X_{k-1,k+2,k+1}-X_{k+2,k+1}X_{k-1,k,k+2}.\] If $t'>k+2$ we obtain \[X_{t',k,k+1}X_{t',k+2}-X_{t',k}X_{t',k+2,k+1}-X_{t',k+1}X_{t',k,k+2}.\] Both of these expressions are seen to be classical Pl\"ucker relations.

In the case $s=2$ the elements $a\land b$, $\tau(p_1(c,d))$, $a$, $b$, $a\lor b$ and $\tau(q_1(c,d))$ are described similarly for both possibilities in Proposition~\ref{specialpairs}. The only difference is that we do not need to consider the $(s-2)$nd subscript in these descriptions, i.e.\ all subscripts coincide outside of positions $s-1$ and $s$. In particular, there is no value $t'$ in this case and we do not need to consider different cases depending on this value. 
\end{proof}
This is how the above proposition can be summed up as a subgraph in the Hasse diagram of $\mN$.
\begin{center}
\begin{tikzcd}[row sep=1mm, column sep=1mm]
&\tau(p_1(c,d))=h_1(a,b)&\\[4pt]
&a\lor b\arrow[u]&\\
 a\arrow[ru]&& b\arrow[lu]\\[3pt]
 &\arrow[lu] a\land b=g_1(a,b)\arrow[ru]&\\[4pt]
 &\tau(q_1(c,d))=a\odot b\arrow[u]&\\
\end{tikzcd}
\end{center}

Next, we will also need two facts concerning non-special diamond pairs.
\begin{proposition}\label{odotland}
If a diamond pair $\{a,b\}\subset\mN$ is not special, then $a\odot b=a\land b$.
\end{proposition}
\begin{proof}
We are to prove the inclusion $J=\iota_\mN(a\land b)\subset\iota_\mN(a\odot b)$ or, in other words, that every maximal element of $J=\iota_\mN(a)\cap\iota_\mN(b)$ is contained in $\iota_\mN(a\odot b)$. Now, if $x$ is maximal in $J$ and is also maximal in one of $\iota_\mN(a)$ and $\iota_\mN(b)$, then from equation~\eqref{indicator} and Proposition~\ref{star} we see that $x$ must lie in $K_{V_o,V_c}(\iota_\mN(a\odot b))\subset\iota_\mN(a\odot b)$. Indeed, if $x$ lies in both $K_{V_o,V_c}(\iota_\mN(a))$ and $K_{V_o,V_c}(\iota_\mN(b))$, then~\eqref{indicator} immediately implies that it also lies in $K_{V_o,V_c}(\iota_\mN(a\odot b))$. If it lies in just one of them, say $K_{V_o,V_c}(\iota_\mN(a))$, we apply Proposition~\ref{combdesc}. We see that $x$ may not lie in $V_o$, since $J\cap V_o=K_{V_o,V_c}(\iota_\mN(a))\cap K_{V_o,V_c}(\iota_\mN(b))\cap V_o$. This implies that $x$ is not maximal in $\iota_\mN(b)$ and, therefore, it does not lie in $K_{V_o,V_c}(\iota_\mN(a)\cup\iota_\mN(b))$. Now~\eqref{indicator} again shows that $x\in K_{V_o,V_c}(\iota_\mN(a\odot b))$. Thus it suffices to show that any maximal element of $J$ is also maximal in one of $\iota_\mN(a)$ and $\iota_\mN(b)$.

We have $\iota_\mN(a)=J\cup\{x_{s,t}\}$ and $\iota_\mN(b)=J\cup\{x_{u,v}\}$ where $u<s\le t<v$. The fact that $\{a,b\}$ is not special means that $s-u>1$ or $v-t>1$. If $x_{\rho,\sigma}$ is maximal in $J$, then the set $\{x_{\rho+1,\sigma},x_{\rho,\sigma+1}\}$ does not contain at least one of $x_{s,t}$ and $x_{u,v}$. If it does not contain the former, then $x_{\rho,\sigma}$ is maximal in $\iota_\mN(a)$, if the latter, then it is maximal in $\iota_\mN(b)$.
\end{proof}

\begin{proposition}
Let diamond pair $\{a,b\}\subset\mN$ be non-special and consider some $i\in[1,m(a,b)]$. Then there exists $c\notin\{a,b\}$ that is incomparable to at least one of $a$ and $b$ and such that $g_i(a,b)<c<h_i(a,b)$. 
\end{proposition}
\begin{proof}
Let $\iota_\mN(a)=\iota_\mN(a\land b)\cup\{x_{s,t}\}$ and $\iota_\mN(b)=\iota_\mN(a\land b)\cup\{x_{u,v}\}$ where $u<s\le t<v$. Denote $J=\tau(g_i(a,b))$. Let $J_1$ be the union of $J$ and the set of all $x\le x_{s,t}$ and let $J_2$ be the union of $J$ and the set of all $x\le x_{u,v}$. Evidently, $J\subsetneq J_1$ and $J\subsetneq J_2$, since $J$ contains neither of $x_{s,t}$ and $x_{u,v}$. It is also evident that $J_1\subsetneq\iota_\mN(a\lor b)$ and $J_2\subsetneq\iota_\mN(a\lor b)$, since $x_{u,v}\notin J_1$ and $x_{s,t}\notin J_2$. Furthermore, since $x_{s,t}\in J_1$ and $x_{s,t}\notin\iota_\mN(b)$ while $x_{v,u}\notin J_1$ and $x_{v,u}\in\iota_\mN(b)$ we see that $\iota_\mN^{-1}(J_1)$ is incomparable to $b$. Similarly $\iota_\mN^{-1}(J_2)$ is incomparable to $a$. It remains to show that either $J_1\neq \iota_\mN(a)$ or $J_2\neq \iota_\mN(b)$, then we can take one of $\iota_\mN^{-1}(J_1)$ and $\iota_\mN^{-1}(J_2)$ as the desired $c$.

Note that $J\subsetneq\iota_\mN(a\land b)$, otherwise, in view of the previous proposition, we would have $g_i(a,b)=g_0(a,b)$. Let $K=\iota_\mN(a\land b)\backslash J$. Suppose that $J_1=\iota_\mN(a)$ and $J_2=\iota_\mN(b)$. This means that $K$ is contained in the set of $x\le x_{s,t}$ and is also contained in the set of $x\le x_{u,v}$. In other words, any $x\in K$ satisfies $x\le x_{u,t}$. However, the difference $K$ must contain at least one element that is maximal in $\iota_\mN(a\land b)$. Since $x_{u,t}\in\iota_\mN(a\land b)$, no $x\in K$ with $x\neq x_{u,t}$ can be maximal in $\iota_\mN(a\land b)$. Consequently, $x_{u,t}$ must be maximal in $\iota_\mN(a\land b)$. However, since $s-u>1$ or $v-t>1$, the elements $x_{s-1,t}$ and $x_{u,v-1}$ are distinct. Both $x_{s-1,t}\ge x_{u,t}$ and $x_{u,v-1}\ge x_{u,t}$ and, therefore, at least one of them is strictly greater than $x_{u,t}$ which contradicts $x_{u,t}$ being maximal in $\iota_\mN(a\land b)$.
\end{proof}
\begin{remark}
With some more effort one can, in fact, show that for a non-special diamond pair $\{a,b\}\subset M$ we also have $m(a,b)=1$ and give explicit descriptions of $g_1(a,b)$ and $h_1(a,b)$. However, we avoid going into more combinatorial detail than necessary to prove the main result below.
\end{remark}

We are ready to prove the minimal H-description of $C(I,I^m(\mN))$.
\begin{theorem}\label{mainpbw}
The cone $C(I,I^m(\mN))$ consists of those $w\subset\bR^\mN$ that satisfy \[w_a+w_b<w_{a\lor b}+w_{a\odot b}\] for every diamond pair $\{a,b\}\subset\mN$ and satisfy \[w_a+w_b<w_{g_1(a,b)}+w_{h_1(a,b)}\] for every special diamond pair $\{a,b\}\subset\mN$. This H-description is minimal.
\end{theorem}
\begin{proof}
The proof is similar to that of Theorem~\ref{maingt}. Let $\{a,b\}$ be an incomparable pair which is not a special diamond pair and consider $i\in[0,m(a,b)]$ such that $i>0$ if $\{a,b\}$ is a diamond pair. If the hyperplane $w_a+w_b=w_{g_i(a,b)}+w_{h_i(a,b)}$ contains a facet of $C(I,I^m(\mN))$, then there must exist a $u\in\bR^\mN$ such that \[u_a+u_b>u_{g_i(a,b)}+u_{h_i(a,b)}\] but \[u_c+u_d<u_{g_j(c,d)}+u_{h_j(c,d)}\] when $\{c,d\}\neq\{a,b\}$ or $j\neq i$. With the use of the last proposition we produce a $c$ which lies strictly between $g_i(a,b)$ and $h_i(a,b)$, is incomparable to at least one of $a$ and $b$ and is neither $a$ nor $b$. We then denote $\deg X_aX_bX_c=\la$ and $\wt X_aX_bX_c=\mu$ and show that the dimensions of the homogeneous components $\dim I_{\la,\mu}$ and $\dim (\initial_u I)_{\la,\mu}$ must be different, achieving a contradiction just as in the proof of Theorem~\ref{maingt}.

We have shown that every facet is given by one of the inequalities in the statement of the theorem, we are left to show that each of these inequalities does indeed provide a facet. In view of Proposition~\ref{pbwdiamondfacets}, it suffices to show that for a chosen special diamond pair $\{a,b\}\subset\mN$ there exists a point $v$ such that 
\begin{equation}\label{pbwreverse}
v_a+v_b> v_{g_1(a,b)}+v_{h_1(a,b)}
\end{equation}
but 
\begin{equation}\label{pbwnonstrict}
v_c+v_d\le v_{g_i(c,d)}+v_{h_i(c,d)}
\end{equation}
whenever $\{c,d\}\subset\mN$ is a diamond pair and $i=0$ or $\{c,d\}\neq\{a,b\}$ is a special diamond pair and $i=1$. Recall the point $v\in\bR^\mM$ which was defined in the proof of Theorem~\ref{maingt}. To avoid conflicting notations denote this latter point $\hat v$ and let $\{\tau^{-1}(a),\tau^{-1}(b)\}\subset\mM$ be the special diamond pair with respect to which $\hat v$ was defined (i.e\ the pair denoted $\{a,b\}$ in the proof of Theorem~\ref{maingt}). It turns out that we may simply set $v_c=\hat v_{\tau^{-1}(c)}$ for all $c\in\mN$.

Indeed, inequality~\eqref{pbwreverse} reduces to $2>1$. When $\{c,d\}=\{a,b\}$ and $i=0$ inequality~\eqref{pbwnonstrict} reduces to $2\ge 2$. Furthermore, inequality \eqref{pbwnonstrict} is evident whenever $|g_i(c,d)|<|a\odot b|$ or $|h_i(c,d)|>|h_1(a,b)|$. If $\{c,d\}$ is a non-special diamond pair and $i=0$, then, in view of Proposition~\ref{odotland}, the corresponding inequality is verified in cases 2--4 in the proof of Theorem~\ref{maingt}. We are left to consider four cases in all of which $\{c,d\}\neq\{a,b\}$ is a special diamond pair (here it should be helpful to consult the diagram in the proof of Theorem~\ref{maingt}).
\begin{enumerate}
\item We have $|c|=|d|=|a\land b|$ and $i=1$. Unless $g_i(c,d)=a\odot b$, we have $v_{g_i(c,d)}=2$ and \eqref{pbwnonstrict} follows. If $g_i(c,d)=a\odot b$, we must show that we do not have $v_{h_i(c,d)}=0$. However, $v_{h_i(c,d)}=0$ would imply $h_i(c,d)<h_1(a,b)$ and due to Corollary~\ref{only4} we would have $\{c,d\}\subset\{a,b,a\land b,a\lor b\}$\ which is impossible. 
\item We have $|c|=|d|=|a|$ and $i=0$. Just as in the previous case we have $v_{g_i(c,d)}=2$ unless $g_i(c,d)=a\odot b$. In the latter case we may not have $v_{h_i(c,d)}=0$, since that would imply $\{c,d\}=\{a,b\}$.
\item We have $|c|=|d|=|a|$ and $i=1$. Then $v_{h_i(c,d)}=2$ unless $h_i(c,d)=h_1(a,b)$. Suppose that $h_i(c,d)=h_1(a,b)$. Since $|g_i(c,d)|=|a\land b|$ and $v_{g_i(c,d)}=1$ we are left to show that we can't have $v_c=v_d=1$. Since $c,d\le h_1(a,b)$, the condition $v_c=v_d=1$ would imply that $c,d<a\lor b$, i.e.\ $c\lor d=a\lor b$. This means that for special diamond pairs $\{\tau^{-1}(a),\tau^{-1}(b)\}\subset\mM$ and $\{\tau^{-1}(c),\tau^{-1}(d)\}\subset\mM$ have $\tau^{-1}(a)\lor \tau^{-1}(b)=\tau^{-1}(c)\lor \tau^{-1}(d)$ and $q_1(\tau^{-1}(a),\tau^{-1}(b))=q_1(\tau^{-1}(c),\tau^{-1}(d))$. However, Proposition~\ref{specialpairs} shows that the values of $\lor$ and $q_1$ determine a special diamond pair in $\mM$ uniquely.
\item We have $|c|=|d|=|a\lor b|$ and $i=0$. Again $v_{h_i(c,d)}=2$ unless $h_i(c,d)=h_1(a,b)$. If $h_i(c,d)=h_1(a,b)$, then $c,d<h_1(a,b)$ and at least one of $v_c$ and $v_d$ is $0$ while $v_{g_i(c,d)}$=1. Inequality~\eqref{pbwnonstrict} follows.\qedhere
\end{enumerate}
\end{proof}

\begin{remark}
We see that the cones $C(I,I^m(\mM))$ and $C(I,I^m(\mN))$ have the same number of facets. However, it is unclear whether these cones are unimodularly or even combinatorially equivalent when $n\ge 4$. This could be an interesting question for investigation. 
\end{remark}

\section{A convex geometric interpretation}

The goal of this section is to show that the described cones $C(I,I^m(\mM))$ and $C(I,I^m(\mN))$ possess an interesting convex geometric property. This property provides a connection between our descriptions of these maximal cones and descriptions of the smaller cones $C(I,I^h(\mM))$ and $C(I,I^{V_o,V_c}(\mN))$ that were given in~\cite{M} and~\cite{fafefom}. Let us recall these results. 

First, we consider the ${n+1}\choose 2$-dimensional space $\Xi=\bR^{\{1\le s\le t\le n\}}$ where $z\in\Xi$ has coordinates $z_{s,t}$. We define a $n\choose 2$-dimensional relatively open cone $\mathcal K\subset\Xi$ composed of $z$ such that
\begin{enumerate}[label=(\roman*)]
\item for any $1\le s\le n$ we have $z_{s,s}=0$ and
\item for any pair $1\le s<t\le n-1$ we have \[z_{s,t}+z_{s+1,t+1}<z_{s,t+1}+z_{s+1,t}.\]
\end{enumerate}
It is easily verified (also see~\cite{M}) that $\mathcal K$ is the product of an $(n-1)$-dimensional real space and a ${n-1}\choose 2$-dimensional simplicial cone. Thus the ${n-1}\choose 2$-inequalities in (ii) provide the facets of $\overline{\mathcal K}$.   

\begin{theorem}[{\cite[Theorem 6.2]{M}}]\label{gttropical}
The cone $C(I,I^h(\mM))$ is the image $\sigma(\mathcal K\times\bR^{n-1})$ where the map $\sigma:\Xi\times\bR^{n-1}\to\bR^\mM$ is defined as follows. For $z\in\Xi$ and $c\in\bR^{n-1}$ the coordinate of $\sigma((z,c))$ corresponding to $a_{i_1,\ldots,i_k}$ with $i_1<\dots<i_k$ is equal to the sum $z_{1,i_1}+\dots+z_{k,i_k}+c_k$.
\end{theorem}

\begin{theorem}[{\cite[Theorem 7.3]{fafefom}}]\label{pbwtropical}
The cone $C(I,I^{V_o,V_c}(\mN))$ is the image $\rho(\mathcal K\times\bR^{n-1})$ where the map $\rho:\Xi\times\bR^{n-1}\to\bR^\mM$ is defined as follows. For $z\in\Xi$ and $c\in\bR^{n-1}$ the coordinate of $\rho((z,c))$ corresponding to $b_{\alpha_1,\ldots,\alpha_k}$ with $(\alpha_1,\dots,\alpha_k)$ a one-column PBW-semistandard tableau is equal to $-z_{1,i_1}-\dots-z_{k,i_k}+c_k$.
\end{theorem}

It should be noted that~\cite{M} (resp.~\cite{fafefom}), in fact, describes the intersection of the $C(I,I^h(\mM))$ (resp. $C(I,I^{V_o,V_c}(\mN))$) with the subspace $w_{a_1}=w_{a_{1,2}}=\dots=w_{a_{1,\dots,n-1}}=0$ (resp. $w_{b_1}=w_{b_{1,2}}=\dots=w_{b_{1,\dots,n-1}}=0$). In the papers these intersections are characterized as the images $\sigma(\mathcal K\times\{0\})$ and $\rho(\mathcal K\times\{0\})$. However, for a point $w\in\bR^\mM$ adding some value $c_k$ to all coordinates with $k$ subscripts does not change $\initial_w I$, due to $I$ being $\deg$-homogeneous, the same goes for $w\in\bR^\mN$. This immediately provides the above reformulations.

Recall that $\overline{C(I,I^h(\mM))}$ is a face of $\overline{C(I,I^m(\mM))}$ and $\overline{C(I,I^{V_o,V_c}(\mN))}$ is a face of $\overline{C(I,I^m(\mN))}$. The mentioned convex geometric properties are as follows.
\begin{theorem}\label{gtconvex}
Every facet of $\overline{C(I,I^m(\mM))}$ either contains $\overline{C(I,I^h(\mM))}$ or intersects $\overline{C(I,I^h(\mM))}$ in a facet of the latter cone.
\end{theorem}
\begin{proof}
Consider the facet of $\overline{C(I,I^m(\mM))}$ contained in the hyperplane $H=\{w_a+w_b=w_{p_i(a,b)}+w_{q_i(a,b)}\}$ where either $\{a,b\}\subset\mM$ is a diamond pair and $i=0$ or $\{a,b\}$ is a special diamond pair and $i=1$. Denote $a=a_{i_1,\dots,i_k}$ and $b=a_{j_1,\dots j_l}$ where either $k=l$ or $k=l+1$. 

First consider the case when $i=0$. We claim that $H$ contains the cone $C(I,I^h(\mM)$. Indeed, consider $w=\sigma((z,c))\in C(I,I^h(\mM)$ as in Theorem~\ref{gttropical}. We see that \[w_a+w_b=z_{1,i_1}+\dots+z_{k,i_k}+c_k+z_{1,j_1}+\dots+z_{l,j_l}+c_l\] and 
\begin{multline*}
w_{p_i(a,b)}+w_{q_i(a,b)}=w_{a\land b}+w_{a\lor b}=z_{1,\min(i_1,j_1)}+\dots+z_{l,\min(i_l,j_l)}+\\z_{l+1,i_{l+1}}+\dots+z_{k,i_k}+c_k+z_{1,\max(i_1,j_1)}+\dots+z_{l,\max(i_l,j_l)}+c_l,
\end{multline*}
the claim follows. (In fact, the same argument shows that the similarly defined hyperplane $H$ contains $C(I,I^h(\mM))$ for any incomparable $\{a,b\}$ when $i=0$.)

Now, suppose that $i=1$ and $\{a,b\}$ is a special diamond pair. We show that $H$ intersects $\overline{C(I,I^h(\mM))}$ in a facet. Let $\Xi_0\subset\Xi$ be the subspace $z_{1,1}=\dots=z_{n,n}=0$. The restriction of $\sigma$ to $\Xi_0\times\bR^{n-1}$ is seen to be injective, therefore, each facet of $\overline{C(I,I^h(\mM))}$ has the form  $\sigma(f\times\bR^{n-1})$ where $f$ is a facet of $\overline{\mathcal K}$.

Denote $p_1(a,b)=a_{\alpha_1,\dots,\alpha_k}$ and $q_1(a,b)=a_{\beta_1,\dots,\beta_l}$. For $w=\sigma((z,c))\in \overline{C(I,I^h(\mM)}$ (where $z\in\overline{\mathcal K}$ and $c\in\bR^{n-1}$) consider the difference $Q=w_a+w_b-w_{p_1(a,b)}-w_{q_1(a,b)}$. If $\{a,b\}$ is given by possibility (1) in Proposition~\ref{specialpairs} so that $i_s=j_s-1=j_{s+1}-2=i_{s+1}-3$, then \[Q=\sum_{r=1}^k(z_{r,i_r}+z_{r,j_r}-z_{r,\alpha_r}-z_{r,\beta_r})=z_{s,i_s+1}-z_{s,i_s+2}+z_{s+1,i_s+2}-z_{s+1,i_s+1}.\] If $\{a,b\}$ is given by possibility (2) in Proposition~\ref{specialpairs} so that $i_k=i_{k-1}+1=j_{k-1}+2=n$, then \[Q=\sum_{r=1}^{k-1}(z_{r,i_r}+z_{r,j_r}-z_{r,\alpha_r}-z_{r,\beta_r})+z_{k,i_k}-z_{k,\alpha_k}=z_{k-1,n-1}-z_{k-1,n}+z_{k,n}-z_{k,n-1}.\] In both cases we see that those points in $\overline{C(I,I^h(\mM)}$ where $Q$ vanishes form a facet of $\overline{C(I,I^h(\mM)}$.
\end{proof}

\begin{theorem}\label{pbwconvex}
Every facet of $\overline{C(I,I^m(\mN))}$ either contains $\overline{C(I,I^{V_o,V_c}(\mN))}$ or intersects $\overline{C(I,I^{V_o,V_c}(\mN))}$ in a facet of the latter cone.
\end{theorem}
\begin{proof}
Consider the facet of $\overline{C(I,I^m(\mN))}$ contained in the hyperplane $H=\{w_a+w_b=w_{g_i(a,b)}+w_{h_i(a,b)}\}$ where either $\{a,b\}\subset\mN$ is a diamond pair and $i=0$ or $\{a,b\}$ is a special diamond pair and $i=1$. Denote $a=b_{\alpha_1,\dots,\alpha_k}$ and $b=b_{\beta_1,\dots \beta_l}$, $g_i(a,b)=b_{\gamma_1,\dots,\gamma_k}$ and $h_i(a,b)=b_{\delta_1,\dots,\delta_l}$ with $k\le l$. For $w=\sigma((z,c))\in \overline{C(I,I^h(\mM)}$ (where $z\in\overline{\mathcal K}$ and $c\in\bR^{n-1}$) consider the difference $Q=w_a+w_b-w_{g_i(a,b)}-w_{h_i(a,b)}$.

We claim that if $i=0$, then necessarily $Q=0$. Indeed, in this case $g_i(a,b)=a\odot b$ and $h_i(a,b)=a\lor b$. The set $\{x_{r,\alpha_r}|\alpha_r>r\}\subset\mathcal P(\mN)$ is the set of maximal elements in $\iota_\mN(a)$ that lie in $V_c$. The sets $\{x_{r,\beta_r}|\beta_r>r\}$, $\{x_{r,\gamma_r}|\gamma_r>r\}$ and $\{x_{r,\delta_r}|\delta_r>r\}$ are characterized similarly. From the definition of $\odot$ we now see that $Q=0$.

Now suppose that $i=1$ and $\{a,b\}$ is a special diamond pair. Denote $\iota_\mN(a)\backslash\iota_\mN(a\land b)=x_{s,t}$. Now, suppose that the special diamond pair $\{\tau^{-1}(a),\tau^{-1}(b)\}\subset\mM$ is given by possibility (1) in Proposition~\ref{specialpairs}. In the proof of Proposition~\ref{pbwspecialpairs} we have seen that $k=l$ and $\alpha_r=\beta_r=\gamma_r=\delta_r$ unless $r\in[s-2,s]$ while $\alpha_{s-2}=\gamma_{s-2}$, $\beta_{s-2}=\delta_{s-2}$, $\alpha_{s-1}=\delta_{s-1}$ and $\beta_{s}=\gamma_{s}$. There we have also seen that $\beta_{s-1}=\delta_s=t+1$ and $\alpha_s=\gamma_{s-1}=t$. We obtain \[Q=-z_{s-1,t+1}-z_{s,t}+z_{s-1,t}+z_{s,t+1}.\]

If $\{\tau^{-1}(a),\tau^{-1}(b)\}\subset\mM$ is given by possibility (1) in Proposition~\ref{specialpairs}, then by similarly consulting the proof of Proposition~\ref{pbwspecialpairs} we obtain \[Q=-z_{k,k+2}-z_{k+1,k+1}+z_{k,k+1}+z_{k+1,k+2}.\]

Again, in both cases we see that those points in $\overline{C(I,I^h(\mM)}$ where $Q$ vanishes form a facet of $\overline{C(I,I^h(\mM)}$.
\end{proof}

\begin{remark}
By applying Theorems~\ref{gtconvex} and~\ref{pbwconvex} one could make the proofs of the main Theorems~\ref{maingt} and~\ref{mainpbw} somewhat simpler and more natural. Therefore, it would be interesting to obtain proofs of Theorems~\ref{gtconvex} and~\ref{pbwconvex} which do not rely on the main theorems. No such proofs have been found as of yet despite several attempts. 
\end{remark}

\begin{remark}
Theorems~\ref{gtconvex} and~\ref{pbwconvex} show that both $\overline{C(I,I^m(\mM))}$ and $\overline{C(I,I^m(\mN))}$ have a curious property. There exists a proper face $F$ such that every facet either contains $F$ or intersects $F$ in a facet of $F$. This seems to be a reasonably strong property, most convex polyhedra do not have such faces. A simple example of a polyhedron and its face which do have this property is a pyramid and its base. It would be interesting to see if one could deduce any further information about the convex geometry of $C(I,I^m(\mM))$ and $C(I,I^m(\mN))$ from this property.
\end{remark}

\end{document}